\documentclass[11pt,reqno]{amsart}

\usepackage{amsmath, amsthm, amsopn, amssymb}
\usepackage[nobysame,msc-links]{amsrefs}
\usepackage{enumerate, tikz, etoolbox, intcalc, geometry, caption, subcaption, afterpage}
\usepackage[english]{babel}
\newtheorem{thm}{Theorem}[section]
\newtheorem{lemma}[thm]{Lemma}
\newtheorem{prop}[thm]{Proposition}

\newtheorem{cor}[thm]{Corollary}

\theoremstyle{definition}
\newtheorem{remark}{Remark}

\newcommand{\cA}{\mathcal{A}}

\newcommand{\cC}{\mathcal{C}}

\newcommand{\cL}{\mathcal{L}}
\newcommand{\cF}{\mathcal{F}}
\newcommand{\cE}{\mathcal{E}}

\newcommand{\cS}{\mathcal{S}}
\newcommand{\cT}{\mathcal{T}}
\newcommand{\cI}{\mathcal{I}}
\newcommand{\cW}{\mathcal{W}}

\newcommand{\N}{\mathbb{N}}
\newcommand{\R}{\mathbb{R}}
\newcommand{\Z}{\mathbb{Z}}

\newcommand{\E}{\mathbb{E}}

\renewcommand{\Pr}{\mathbb{P}}
\newcommand{\1}{\mathbf{1}}

\newcommand{\zero}{\mathbf{0}}

\DeclareMathOperator\diam{diam}
\DeclareMathOperator\dist{dist}

\newcommand{\iid}{i.i.d.}
\newcommand{\ffiid}{ffiid}
\newcommand{\fvffiid}{fv-ffiid}

\title{Finitary codings for spatial mixing Markov random fields}
\date{\today}

\author{Yinon Spinka}
\address{Tel Aviv University.
    School of Mathematical Sciences.
    Tel Aviv, 69978, Israel.}
    \email{yinonspi@post.tau.ac.il}
\thanks{Research supported by Israeli Science Foundation grant 861/15, the European Research Council starting grant 678520 (LocalOrder), and the Adams Fellowship Program of the Israel Academy of Sciences and Humanities}

\begin{document}

\begin{abstract}
	It has been shown by van den Berg and Steif~\cite{van1999existence} that the sub-critical and critical Ising model on~$\Z^d$ is a finitary factor of an \iid\ process (\ffiid), whereas the super-critical model is not. In fact, they showed that the latter is a general phenomenon in that a phase transition presents an obstruction for being \ffiid. The question remained whether this is the only such obstruction. We make progress on this, showing that certain spatial mixing conditions (notions of weak dependence on boundary conditions, not to be confused with other notions of mixing in ergodic theory) imply \ffiid.
	Our main result is that weak spatial mixing implies \ffiid\ with power-law tails for the coding radius, and that strong spatial mixing implies \ffiid\ with exponential tails for the coding radius. The weak spatial mixing condition can be relaxed to a condition which is satisfied by some critical two-dimensional models. Using a result of the author~\cite{spinka2018finitary}, we deduce that strong spatial mixing also implies \ffiid\ with stretched-exponential tails from a \emph{finite-valued} \iid\ process.
	
	We give several applications to models such as the Potts model, proper colorings, the hard-core model, the Widom--Rowlinson model and the beach model. For instance, for the ferromagnetic $q$-state Potts model on $\Z^d$ at inverse temperature~$\beta$, we show that it is \ffiid\ with exponential tails if $\beta$ is sufficiently small, it is \ffiid\ if $\beta < \beta_c(q,d)$, it is not \ffiid\ if $\beta > \beta_c(q,d)$ and, when $d=2$ and $\beta=\beta_c(q,d)$, it is \ffiid\ if and only if $q \le 4$.
\end{abstract}

\maketitle

\section{Introduction and Results}\label{sec:introduction}

Let $(S,\cS)$ and $(T,\cT)$ be two measurable spaces, and let $X=(X_v)_{v \in \Z^d}$ and $Y=(Y_v)_{v \in \Z^d}$ be $(S,\cS)$-valued and $(T,\cT)$-valued translation-invariant random fields (i.e., stationary $\Z^d$-processes) for some $d \ge 1$.
A \emph{coding} from $Y$ to $X$ is a measurable function $\varphi \colon T^{\Z^d} \to S^{\Z^d}$, which is \emph{translation-equivariant}, i.e., commutes with every translation of $\Z^d$, and which satisfies that $\varphi(Y)$ and $X$ are identical in distribution. Such a coding is also called a \emph{factor map} or \emph{homomorphism} from $Y$ to $X$, and when such a coding exists, we say that $X$ is a \emph{factor} of $Y$.

The \emph{coding radius} of $\varphi$ at a point $y \in T^{\Z^d}$, denoted by $R_\varphi(y)$, is the minimal integer $r \ge 0$ such that $\varphi(y')_\zero=\varphi(y)_\zero$ for all $y' \in T^{\Z^d}$ which coincide with $y$ on the ball of radius $r$ around the origin in the graph-distance, i.e., $y'_v=y_v$ for all $v \in \Z^d$ such that $\|v\|_1 \le r$. It may happen that no such~$r$ exists, in which case, $R_\varphi(y)=\infty$. Thus, associated to a coding is a random variable $R=R_\varphi(Y)$ which describes the coding radius. A coding is called \emph{finitary} if $R$ is almost surely finite. When there exists a finitary coding from $Y$ to $X$, we say that $X$ is a \emph{finitary factor} of $Y$. When $X$ is a finitary factor of an \iid\ (independent and identically distributed) process, we say it is \emph{\ffiid}. We say that a coding has \emph{exponential tails} if $\Pr(R \ge r) \le Ce^{-cr}$ for some $C,c>0$ and all $r>0$ and that it has \emph{power-law tails} if $\Pr(R \ge r) \le Cr^{-c}$ instead.

In this paper, we always take $S$ to be a non-empty finite set (and $\cS$ the power set of $S$) and $X$ to be a (finite-valued) Markov random field.
With the relevant definitions given below, we state our main result.

\begin{thm}\label{thm:main}
	Let $d \ge 2$ and let $X$ be a translation-invariant Markov random field on $\Z^d$.
	\begin{enumerate}[\quad~~(a)]
		\item\label{thm:main-a} If $X$ satisfies exponential weak spatial mixing, then it is \ffiid\ with power-law tails.
		\item\label{thm:main-b} If $X$ satisfies exponential strong spatial mixing, then it is \ffiid\ with exponential tails.
	\end{enumerate}
\end{thm}

Let $X=(X_v)_{v \in \Z^d}$ be a random field.
We say that $X$ is a \emph{Markov random field} if the conditional distribution of its values on any finite set $V \subset \Z^d$, given its values on $V^c$, depends only on its values on the boundary $\partial V$, i.e., if $(X_v)_{v \in V}$ and $(X_v)_{v \in V^c}$ are conditionally independent given $(X_v)_{v \in \partial V}$, where $\partial V$ is the set of vertices at distance 1 from~$V$.
We say that $X$ is \emph{translation-invariant} if its distribution is not affected by translations, i.e., if $(X_{v+u})_{v \in V}$ has the same distribution as $X$ for all $u \in \Z^d$.

Let $X$ be a Markov random field.
Denote the \emph{topological support} of $X$ by
\begin{equation}\label{eq:def-Omega}
\Omega := \Big\{ x \in S^{\Z^d} : \Pr(X_V=x_V)>0\text{ for all finite }V \subset \Z^d \Big\} ,
\end{equation}
where we write $x_V := (x_v)_{v \in V}$ for the restriction of a configuration $x \in S^{\Z^d}$ to $V$.
Thus, $\Omega$ is the set of ``feasible'' configurations.
Note that $\Omega$ is a closed subset of $S^{\Z^d}$ in the product topology.
When $X$ is translation-invariant, $\Omega$ is a \emph{shift space} (a closed invariant subset of $S^{\Z^d}$).
For $V \subset \Z^d$ and $\tau \in S^{\partial V}$, denote
\[ \Omega_V := \big\{ \omega_V : \omega \in \Omega \big\} \qquad\text{and}\qquad \Omega_V^\tau := \big\{ \omega_V : \omega \in \Omega,~\omega_{\partial V} = \tau \big\} .\]
The \emph{specification} of $X$ is the collection $P$ of conditional finite-dimensional distributions, namely,
\[ P:=(P^\tau_V)_{V\subset \Z^d\text{ finite},\,\tau \in \Omega_{\partial V}} \qquad\text{and}\qquad P^\tau_V := \Pr\big(X_V \in \cdot \mid X_{\partial V} = \tau\big).\]
Note that, for $V \subset \Z^d$ finite and $\tau \in \Omega_{\partial V}$, the measure $P^\tau_V$ is well-defined and supported on $\Omega^\tau_V$.
We write $P^\tau_v$ for $P^\tau_{\{v\}}$, and when $U \subset V$, we write $P^\tau_{V,U}$ for the marginal of $P^\tau_V$ on $U$.

Usually a specification $\bar{P}$ is defined in and of itself for all boundary conditions arising from some closed subset $\bar{\Omega}$ of $S^{\Z^d}$, and one then says that $X$ is a Gibbs measure for $\bar{P}$ if $\Omega \subset \bar\Omega$ and $P^\tau_V=\bar{P}^\tau_V$ for all finite $V \subset \Z^d$ and $\tau \in \Omega_{\partial V}$ (i.e., $P$ agrees with $\bar{P}$ for all feasible boundary conditions). This point of view will not be important for us and so we chose only to define the ``specification of $X$''.
We also remark that while the specification of a translation-invariant Markov random field is always translation-invariant (in the obvious sense), the converse is not true -- there exist non-translation-invariant Markov random fields whose specifications are translation-invariant (e.g., the Dobrushin state for the low-temperature Ising model in three dimensions).

A Markov random field and its specification are said to satisfy \emph{weak spatial mixing} with rate function $\rho$ if, for any finite sets $U \subset V \subset \Z^d$,
\begin{equation}\label{eq:WSM-def}
\big\|P^\tau_{V,U} - P^{\tau'}_{V,U} \big\| \le |U| \cdot \rho(\dist(U,\partial V)) \qquad\text{for any }\tau,\tau' \in \Omega_{\partial V},\end{equation}
where $\|\cdot\|$ denotes \emph{total-variation distance} (half the $L^1$ norm) and $\dist(A,B)$ is the minimum graph-distance between a vertex in $A$ and a vertex in $B$.
We use the term \emph{rate function} to refer to any monotone function $\rho \colon \{1,2,\dots\} \to [0,\infty)$ which tends to~0.
We remark that the precise definition of weak spatial mixing varies throughout the literature; our definition closely follows the one in~\cite{weitz2004mixing} (see also~\cite{martinelli1999lectures}).
It is easy to verify that if a specification satisfies weak spatial mixing (with any rate function), then there is a unique Markov random field having that specification (i.e., it has a unique Gibbs measure).
We say that a specification satisfies \emph{exponential weak spatial mixing} if it satisfies weak spatial mixing with an exponentially decaying rate function, i.e., if $\rho(n) \le Ce^{-cn}$ for some $C,c>0$ and all $n \ge 1$.

A Markov random field and its specification are said to satisfy \emph{strong spatial mixing} with rate function $\rho$ if, for any finite sets $U \subset V \subset \Z^d$,
\begin{equation}\label{eq:SSM-def}
\big\|P^\tau_{V,U} - P^{\tau'}_{V,U} \big\| \le |U| \cdot \rho(\dist(U,\Sigma_{\tau,\tau'})) \qquad\text{for any }\tau,\tau' \in \Omega_{\partial V},
\end{equation}
where $\Sigma_{\tau,\tau'}$ is the subset of $\partial V$ on which $\tau$ and $\tau'$ disagree. Evidently, \eqref{eq:SSM-def} implies~\eqref{eq:WSM-def}. We say that a specification satisfies \emph{exponential strong spatial mixing} if it satisfies strong spatial mixing with an exponentially decaying rate function. 
We mention that the definition of strong spatial mixing may vary in the literature, especially in the class of $(U,V)$ for which~\eqref{eq:SSM-def} is required to hold (for example, $V$ may be restricted to be a cube). Though we state the strongest form for simplicity (requiring~\eqref{eq:SSM-def} to hold for arbitrary $(U,V)$), the proof of part~\eqref{thm:main-b} of Theorem~\ref{thm:main} uses~\eqref{eq:SSM-def} only for a restricted class of $(U,V)$.
This is particularly relevant in two-dimensions as it allows us to prove \ffiid\ with exponential tails under a weaker assumption (see Section~\ref{sec:discussion}).

The assumption of exponential weak spatial mixing in part~\eqref{thm:main-a} of Theorem~\ref{thm:main} can be substantially relaxed.
For finite sets $U \subset V \subset \Z^d$, define
\begin{equation}\label{eq:def-gamma}
	\gamma(V,U) := \sum_{\omega \in \Omega_U} \min_{\tau \in \Omega_{\partial V}} P^\tau_{V,U}(\omega) .
\end{equation}
As we will see in Section~\ref{sec:couplings}, $\gamma(V,U)$ is the probability under an optimal coupling that configurations sampled in $V$ with every possible boundary condition all coincide on $U$.
For $r>0$, denote
\[ \Lambda_r := [-r,r]^d \cap \Z^d .\]

\begin{thm}\label{thm:ffiid}
	Let $X$ be a translation-invariant Markov random field.
	If
	\begin{equation}\label{eq:positive-coupling-prob2}
		\limsup_{n \to \infty} \gamma(\Lambda_n,\Lambda_{\delta n}) > 0 \qquad\text{for some }0<\delta<1,
	\end{equation}
	then $X$ is \ffiid. If
	\begin{equation}\label{eq:positive-coupling-prob}
		\liminf_{n \to \infty} \gamma(\Lambda_n,\Lambda_{\delta n}) > 0 \qquad\text{for some }0<\delta<1,
	\end{equation}
	then $X$ is \ffiid\ with power-law tails.
\end{thm}

We will show in Section~\ref{sec:weak-mixing} that exponential weak spatial mixing implies~\eqref{eq:positive-coupling-prob} (in fact, it implies that the limit in~\eqref{eq:positive-coupling-prob} tends to 1 for any $0<\delta<1$) so that Theorem~\ref{thm:ffiid} implies part~\eqref{thm:main-a} of Theorem~\ref{thm:main}.
Although Theorem~\ref{thm:ffiid} is applicable in any dimension, our only real application for it (i.e., when~\eqref{eq:positive-coupling-prob} holds, but exponential weak spatial mixing does not hold) is for the critical two-dimensional Potts model (see Corollary~\ref{cor:potts-coding} below).

The theorems presented above are concerned with sufficient conditions for being \ffiid.
In the other direction, van den Berg and Steif gave a necessary condition~\cite[Theorem~2.1]{van1999existence} by showing that a phase transition presents an obstruction for being \ffiid\ in systems with no hard constraints (i.e., when $\Omega=S^{\Z^d}$). That is, they showed that if there exists more than one translation-invariant Markov random field with a given specification in which no hard constraints are present, then neither random field is \ffiid\ (to be precise, the statement of their theorem is that neither field is a finitary factor of a finite-valued \iid\ process, but the same ideas also show that neither field is \ffiid).
We extend this to models with hard constraints, under a mild condition on the probabilities of feasible boundary conditions (which always holds for systems with no hard constraints).
A random field is \emph{ergodic} if it is translation-invariant and it has no non-trivial translation-invariant events, i.e., every such event occurs with probability 0 or 1.

\begin{thm}\label{thm:no-ffiid}
	Let $X$ and $X'$ be two distinct ergodic Markov random fields with the same specification. Suppose that, for some $\epsilon>0$,
	\begin{equation}\label{eq:no-ffiid-assumption}
	\inf_{T \subset \Omega_{\partial \Lambda_n}:~\Pr(X'_{\partial \Lambda_n} \in T) \ge 1-\epsilon}~ \Pr(X_{\partial \Lambda_n} \in T) \ge e^{-o(n^d)} \qquad\text{as }n \to \infty .
	\end{equation}
	Then $X$ is not \ffiid.
\end{thm}

In our applications, we shall also be concerned with models in which the phase transition is of a different nature, where there is a single translation-invariant Markov random field with the given specification, but multiple periodic ones (as occurs, for instance, in the hard-core model at high fugacity). Thus, we also mention here that \emph{total ergodicity} (i.e., ergodicity with respect to \emph{every} translation) is a necessary condition for being \ffiid. This is due to the fact that an \iid\ process is totally ergodic and that factors preserve ergodicity with respect to any translation.

\begin{remark}
	The statements of our results mention nothing about the \iid\ process involved in the \ffiid\ conclusion.
	However, any result in this paper that shows that a translation-invariant Markov random field $X$ is \ffiid, actually yields a stronger property. Specifically, any coding $\varphi$ from $Y$ to $X$ that we construct has the property that $Y_v$ is a sequence of independent \emph{finite-valued} random variables $(Y_{v,n})_{n \in \N}$, and $\varphi$ is finitary in both the $\Z^d$ space (as in the usual definition of finitary coding) and in the $\N$ space, i.e., there exists an almost surely finite $\tilde{R}$ such that $\varphi(Y)_\zero$ is determined by $(Y_{v,n})_{|v| \le \tilde{R}, n \le \tilde{R}}$. Moreover, for the results whose conclusions are \ffiid\ with exponential tails, the random variables $Y_{v,n}$ are, in addition, identically distributed, and $\tilde{R}$ is shown to have exponential tails. This allows one to apply a result of the author~\cite[Proposition~9]{spinka2018finitary} to obtain that $X$ is also a finitary factor of a \emph{finite-valued} \iid\ process (abbreviated \fvffiid) with stretched-exponential tails, meaning that the coding radius $R$ satisfies $\Pr(R \ge r) \le Ce^{-r^c}$ for some $C,c>0$ and all $r>0$. In particular, every time we write ``\ffiid\ with exponential tails'', we could instead (or in addition) write ``\fvffiid\ with stretched-exponential tails''.
	We chose not to include these comments in the statements of our theorems as we preferred to keep these statements short and concise and did not want to lose sight of the main goal, which is showing that certain Markov random fields are \ffiid.
\end{remark}

\begin{remark}\label{rem:r-Markov}
	All our results extend to $r$-Markov random fields, which are random fields $X$ satisfying that, for any finite $V \subset \Z^d$, $X_V$ and $X_{V^c}$ are conditionally independent given $X_{\partial^r V}$, where $\partial^r V$ is the set of vertices $u \in \Z^d$ having $1 \le \dist(u,V) \le r$.
\end{remark}

\begin{remark}\label{rem:other-graphs}
	We chose to work with the graph $G=\Z^d$ and the group $\Gamma$ of its translations, though much of the proof works in greater generality.
	For instance, the proof shows that if one assumes (in addition to the assumptions stated in the theorems) that $X$ is invariant to all automorphisms of $\Z^d$ (not just translations), then the resulting coding commutes with all automorphisms.
	In a more general setting, $G$ is a transitive locally-finite graph, $\Gamma$ is a group acting transitively on $V(G)$ by automorphisms, $X$ is a $\Gamma$-invariant Markov random field, and a coding is a $\Gamma$-equivariant map.
	Our results (and proofs) readily extend to the triangular lattice, the hexagonal lattice, the line graph of $\Z^d$ and variants of these.
	However, as our proof relies on geometric considerations (namely in Section~\ref{sec:grand-coupling}), it is unclear how wide the class of graphs for which our results hold is.
	In particular, we do not know whether they hold for all transitive graphs of sub-exponential growth (part~\eqref{thm:main-a} of Theorem~\ref{thm:main} seems to easily extend to this class). We also mention that transitivity is not strictly necessary and the definitions and proofs could be adapted to work for, say, quasi-transitive actions on graphs similar to $\Z^d$ (e.g., when $G=\Z^d$ and $\Gamma$ is the group of parity-preserving translations).
\end{remark}

\subsection{Applications}
Many classical discrete spin systems on $\Z^d$ satisfy spatial mixing properties (for appropriate model parameters). In some of the examples below, we refer to Dobrushin's uniqueness condition~\cite{Dobrushin1968TheDe}, to van den Berg and Maes's disagreement percolation condition~\cite{van1994disagreement} (in which case, we denote by $p_c(\Z^d)$ the critical probability for site-percolation on $\Z^d$) and to H{\"a}ggstr{\"o}m and Steif's high noise condition~\cite{haggstrom2000propp}, and use the fact each implies exponential strong spatial mixing. See Section~\ref{sec:discussion} for a discussion about these conditions. We write $Z^\tau_V$ below for a normalizing constant which makes $P^\tau_V$ a probability measure.

\subsubsection{\bf Proper colorings}
Let $q \ge 3$ and set $S:=\{1,\dots,q\}$.
A \emph{proper $q$-coloring} is a configuration $x \in S^{\Z^d}$ satisfying that $x_u \neq x_v$ for any adjacent vertices $u$ and $v$. Let $\Omega$ be the set of proper $q$-colorings of $\Z^d$ and let the specification be uniform, i.e., $P^\tau_V$ is the uniform distribution on $\Omega^\tau_V$, the set of proper $q$-colorings of $V$ which are compatible with the boundary condition $\tau$. It is well-known (e.g., by Dobrushin's uniqueness condition or by disagreement percolation) that there is a unique Gibbs measure for proper $q$-colorings when $q>4d$, and that this measure satisfies exponential strong spatial mixing.
In~\cite{goldberg2005strong}, it is shown that exponential strong spatial mixing holds in any triangle-free graph of maximum degree $\Delta \ge 3$ when $q >\alpha\Delta - \gamma$, where
\begin{equation}\label{eq:hard-core-params}
\alpha\text{ is the solution to }\alpha^\alpha=e \qquad\text{and}\qquad \gamma := \frac{4\alpha^3-6\alpha^2-3\alpha+4}{2(\alpha^2-1)} .
\end{equation}
(so $\alpha \approx 1.763$ and $\gamma \approx 0.47$). It is also shown there that, for $\Z^3$, $q \ge 10$ suffices. For $\Z^2$, it known that $q \ge 6$ suffices~\cite{achlioptas2005rapid,goldberg2006improved}.
On the other hand, for any $q \ge 3$, it is known that there are multiple maximum-entropy Gibbs measures for proper $q$-colorings of $\Z^d$ when the dimension $d$ is sufficiently large ($d \ge Cq^{10}\log^3 q$ suffices, with $C$ a universal constant)~\cite{peled2018rigidity}, and that these are not totally ergodic (they are mixtures of 2-periodic extreme Gibbs measures), so that, in particular, they are not \ffiid.
Thus, we obtain the following corollary.

\begin{cor}\label{cor:proper-colorings}
For $d \ge 2$ and $q>2\alpha d - \gamma$, with $\alpha$ and $\gamma$ given by~\eqref{eq:hard-core-params}, the unique Gibbs measure for proper $q$-colorings of $\Z^d$ is \ffiid\ with exponential tails. Moreover, this also holds for 6-colorings in two dimensions and 10-colorings in three dimensions. In addition, for any $q \ge 3$, when $d$ is sufficiently large, no maximum-entropy Gibbs measure for proper $q$-colorings of $\Z^d$ is \ffiid.
\end{cor}

\begin{remark}
	The first part of the corollary extends a previous result of the author~\cite{spinka2018finitary}, where \ffiid\ with exponential tails was shown when $q \ge 4d(d+1)$.
\end{remark}

\subsubsection{\bf The Potts model}\label{sec:Potts-model}
Let $q \ge 2$ and set $S := \{1,\dots,q\}$.
Given a parameter $\beta \in \R$, called the \emph{inverse temperature}, the specification is defined by
\begin{equation}\label{eq:Potts-spec}
P^\tau_V(\omega) := \frac{1}{Z^\tau_V} \cdot e^{\beta H^\tau_V(\omega)} , \qquad\text{where }H^\tau_V(\omega) := \sum_{\substack{\{u,v\} \in E(\Z^d)\\\{u,v\} \cap V \neq \emptyset}} \1_{\{\omega_u=\omega_v\}} .
\end{equation}
When $\beta>0$, neighboring spins tend to take the same value and the model is said to be \emph{ferromagnetic}. When $\beta<0$, neighboring spins tend to take different values and the model is said to be \emph{anti-ferromagnetic}. In particular, the case of $\beta=-\infty$ reduces to that of proper $q$-colorings.

For the ferromagnetic Potts model, it is well-known~(see, e.g., \cite[Theorem~3.2]{georgii2001random}) that there exists a critical value $\beta_c=\beta_c(q,d) \in (0,\infty)$ such that there is a unique Gibbs measure at inverse temperature $0<\beta<\beta_c$ and multiple ergodic Gibbs measures at inverse temperature $\beta>\beta_c$.
Using the random-cluster representation of the Potts model and the sharp phase transition in the random-cluster model~\cite{duminil2017sharp}, it follows that at inverse temperature $\beta<\beta_c$, the unique Gibbs measure satisfies exponential weak spatial mixing.
In two dimensions, more is known: there is a unique Gibbs measure at the critical $\beta=\beta_c$ if and only if $q \le 4$~\cite{duminil2016discontinuity,duminil2017continuity}, there is a RSW property at criticality when $q \le 4$~\cite{duminil2017continuity}, and exponential weak spatial mixing implies exponential strong spatial mixing for squares (see Section~\ref{sec:discussion}). Putting these results together allows us to show the following.

\begin{cor}\label{cor:potts-coding}
	Let $d \ge 2$ and $q \ge 2$. There exists $\beta_0(q,d) \ge 1/2d$ satisfying $\beta_0(q,d)=\beta_c(q,d)$ if either $q=2$ or $d=2$, such that the following holds.
	Let $\mu$ be a Gibbs measure for the ferromagnetic $q$-state Potts model on $\Z^d$ at inverse temperature $\beta>0$.
	\begin{itemize}
		\item If $\beta<\beta_0(q,d)$ then $\mu$ is \ffiid\ with exponential tails.
		\item If $\beta<\beta_c(q,d)$ then $\mu$ is \ffiid\ with power-law tails.
		\item If $\beta>\beta_c(q,d)$ then $\mu$ is not \ffiid.
		\item If $\beta=\beta_c(q,d)$ and $d=2$, then $\mu$ is \ffiid\ if and only if $q \le 4$.
	\end{itemize}
\end{cor}

\begin{proof}
	Some of the proofs here rely on the well-known connection between the Potts model and the random-cluster model via the Edwards--Sokal coupling, which we now explain (see, e.g., \cite{grimmett2006random} for definitions and properties of the latter model and coupling). Let $V \subset \Z^d$ be finite and let $G$ be the graph with vertex-set $V \cup \partial V$ and edge-set $\{ \{u,v\} \in E(\Z^d) : \{u,v\} \cap V \neq \emptyset \}$.	
	We denote by $\phi_{V,p,q}$ the random-cluster measure (with the parameter $p$ chosen in correspondence to $\beta$, namely, $p := 1 - e^{-\beta}$) on $G$ with wired boundary conditions (meaning that clusters intersecting $\partial V$ do not contribute a factor of $q$ to the weight of a configuration). An element $\eta$ in the configuration space of $\phi_{V,p,q}$ is a spanning subgraph of $G$.
	For a boundary condition $\tau \in \{1,\dots,q\}^{\partial V}$, let $E^\tau$ be the event that the clusters of $u$ and $v$ in $\eta$ are disjoint for any $u,v \in \partial V$ such that $\tau_u \neq \tau_v$. We may obtain a sample $\omega^\tau$ from $P^\tau_V$ by first sampling $\eta^\tau$ from $\phi_{V,p,q}(\cdot \mid E^\tau)$ and then independently assigning each cluster of $\eta^\tau$ a uniform value in $\{1,\dots,q\}$ if it does not intersect $\partial V$, or the deterministic value $\tau_v$ if it contains some $v \in \partial V$. Moreover, since $E^\tau$ is a decreasing event for any $\tau$, the FKG property of $\phi_{V,p,q}$~\cite[Theorem~3.8]{grimmett2006random} implies that $\eta^\tau$ is stochastically dominated by $\eta^1$ for any $\tau$. Thus, by Strassen's theorem~\cite{strassen1965existence}, we may couple all $(\eta^\tau)_\tau$ so that $\eta^\tau \subset \eta^1$ almost surely for all $\tau$. By carrying out the above construction simultaneously for every $\tau$ (namely, by assigning the same value to a cluster of $\eta^\tau$ and a cluster of $\eta^{\tau'}$ whenever the two clusters are identical and do not intersect $\partial V$), we may further couple $(\omega^\tau)_\tau$ on the same probability space as $(\eta^\tau)_\tau$, so that $(\omega^\tau_U)_\tau$ all coincide on the event that $U$ is not connected to $\partial V$ in $\eta^1$ (i.e., the event that no cluster of $\eta^1$ intersects both $U$ and $\partial V$).
	
	We proceed to prove the corollary.
	The second part will follow from part~\eqref{thm:main-a} of Theorem~\ref{thm:main}, once we verify that $\mu$ satisfies exponential weak spatial mixing. Indeed, by the above coupling, we have that $\|P^\tau_{V,U}-P^{\tau'}_{V,U}\|$ is bounded above by the probability that $U$ is connected to $\partial V$ in~$\eta^1$. Since $p<p_c(q,d)$, the probability that any given $u \in U$ is connected to $\partial V$ in $\eta^1$ decays exponentially fast in $\dist(U,\partial V)$~\cite[Theorem~1.2]{duminil2017sharp}, thus implying exponential weak spatial mixing.
		
	The third part follows from Theorem~\ref{thm:no-ffiid} and the definition of $\beta_c(q,d)$.
	
	The last part follows using the results in~\cite{duminil2016discontinuity,duminil2017continuity}.
	Namely, when $q>4$, there are multiple ergodic Gibbs measures~\cite[Theorem~1.1]{duminil2016discontinuity} so that $\mu$ is not \ffiid\ by Theorem~\ref{thm:no-ffiid}. When $q \le 4$, by the Russo--Seymour--Welsh estimates~\cite[Theorems~3 and~4]{duminil2017continuity} and the FKG property, the probability that $\Lambda_{n/4}$ is connected to $\partial \Lambda_n$ in $\eta^1$ is at most $c$ for some constant $0<c<1$. From this and the coupling above (with $U:=\Lambda_{n/4}$ and $V:=\Lambda_n$), it is easy to see that $\gamma(\Lambda_n,\Lambda_{n/4}) \ge 1-c$ so that Theorem~\ref{thm:ffiid} implies that $\mu$ is \ffiid.
	
	It remains to prove the first part of the corollary. By Dobrushin's uniqueness condition and a direct computation, $\mu$ is \ffiid\ with exponential tails when $\beta < 1/2d$. This shows that $\beta_0(q,d)$ can be chosen to satisfy $\beta_0(q,d) \ge 1/2d$.
	The fact that we may take $\beta_0(q,d)=\beta_c(q,d)$ when $q=2$ follows from the results in~\cite{van1999existence} for the Ising model. The fact that we may also take $\beta_0(q,d)=\beta_c(q,d)$ when $d=2$ is due to the fact that, in two dimensions, exponential weak spatial mixing implies exponential strong spatial mixing for squares~\cite{martinelli19942,alexander2004mixing} and this is the only assumption needed for the proof of part~\eqref{thm:main-b} of Theorem~\ref{thm:main} (see Section~\ref{sec:discussion}).
\end{proof}

\begin{remark}
	The result in the case of the Ising model ($q=2$) was already known~\cite{van1999existence}. In fact, since the Ising model is known to undergo a continuous phase transition in any dimension $d \ge 2$~\cite{yang1952spontaneous,aizenman1986critical,aizenman2015random}, the unique Gibbs measure at the critical point $\beta=\beta_c(2,d)$ is also \ffiid.
	It is still unknown whether this unique Gibbs measure is \fvffiid, though it is known that it is not \ffiid\ with finite expected coding volume~\cite{van1999existence}.
\end{remark}

\begin{remark}
	Using the Edwards--Sokal coupling~\cite{swendsen1987nonuniversal,edwards1988generalization}, it is a simple consequence of Corollary~\ref{cor:potts-coding} that the (wired/free) random-cluster measure $\phi_{p,q}$ on $\Z^d$ with $q \in \N$ is \ffiid\ when $p < p_c(q,d)$ (with exponential tails when $p$ is sufficiently small, $q=2$ or $d=2$), and, when $d=2$ and $p=p_c(q,d)$, is \ffiid\ if and only if $q \le 4$.
	Indeed, under this coupling, given the Potts configuration, the state of the edges in the random-cluster configuration are independent.
	In two dimensions, it follows by duality that $\phi_{p,q}$ is \ffiid\ when $p>p_c(q,d)$. In particular, $\phi_{p,q}$ is \ffiid\ for all $p$ when $d=2$ and $q \le 4$.
	
	In a work with Matan Harel~\cite{harel2018finitary}, we study finitary codings for the random-cluster model and show that $\phi_{p,q}$ is \ffiid\ with exponential tails whenever $p<p_c(q,d)$. In particular, it follows that the unique Gibbs measure for the ferromagnetic Potts model is \ffiid\ with exponential tails whenever $\beta<\beta_c(q,d)$.
\end{remark}

The picture for the anti-ferromagnetic Potts model is less complete.
Dobrushin's uniqueness condition implies that there is a unique Gibbs measure for the anti-ferromagnetic $q$-state Potts model on $\Z^d$ at inverse temperature $\beta<0$ whenever $\beta>-c_q/d$ or $q>4d$~\cite{salas1997absence}, where $c_q>0$ is a constant depending only on $q$, and that this Gibbs measure satisfies exponential strong spatial mixing (see also~\cite{yin2015spatial}). Thus, in this regime, the unique Gibbs measure is \ffiid\ with exponential tails.
In~\cite{goldberg2006improved}, the two-dimensional model is investigated and various sufficient conditions are given for exponential strong spatial mixing (thus implying \ffiid\ with exponential tails), but we do not list those here.
In the other direction, for every $q \ge 3$ and $\beta_0<0$, there exists $d_0(q,\beta_0)$ such that there are multiple ergodic Gibbs measures at inverse temperature $\beta \le \beta_0$ in dimension $d \ge d_0(q,\beta)$~\cite{peled2017condition}, and it follows from Theorem~\ref{thm:no-ffiid} that they cannot be \ffiid.

\begin{remark}
The anti-ferromagnetic Ising model ($q=2$, $\beta<0$) is an interesting special case. On the one hand, it is well-known that, on $\Z^d$ (or any other bipartite graph), it is equivalent to the ferromagnetic Ising model (so that, in particular, $-\beta_c(2,d)$ is the critical inverse temperature for the existence of multiple Gibbs measures for the anti-ferromagnetic model). This is because ``flipping'' the spins on one of the sublattices has the effect of changing the sign of $\beta$ in~\eqref{eq:Potts-spec}.
It would therefore seem that one could transform a finitary coding for the ferromagnetic Ising model at inverse temperature $\beta>0$ into a finitary coding for the anti-ferromagnetic Ising model at inverse temperature $-\beta<0$. However, the operation of ``flipping a sublattice'' cannot be carried out in a translation-equivariant manner, so that this does not provide a recipe for constructing a finitary coding for the anti-ferromagnetic model.
Nevertheless, the anti-ferromagnetic model does indeed satisfy the analogous properties as the ferromagnetic model: \ffiid\ with exponential tails when $0>\beta>-\beta_c(2,d)$, \ffiid\ when $\beta=-\beta_c(2,d)$ and not \ffiid\ when $\beta<-\beta_c(2,d)$.	
This is due to the fact that the anti-ferromagnetic Ising model is anti-monotone (see Section~\ref{sec:discussion}; see also Remark~\ref{rem:hard-core-anti-monotone} where this is explained in more detail for the hard-core model).
\end{remark}

\subsubsection{\bf The hard-core model}\label{sec:hard-core-model}
We set $S:=\{0,1\}$ and identify configurations in $S^{\Z^d}$ with subsets of $\Z^d$.
Let $\Omega$ be the collection of independent sets in $\Z^d$, where an \emph{independent set} is a subset of vertices that contains no two adjacent vertices.
Given a parameter $\lambda>0$, called the \emph{fugacity} or \emph{activity}, the specification is defined by
\[ P^\tau_V(\omega) := \frac{\1_{\{\omega \cup \tau \text{ is an independent set}\}}}{Z^\tau_V} \cdot \lambda^{|\omega|} .\]
Dobrushin's uniqueness condition (or, alternatively, the high noise condition) implies that there is a unique Gibbs measure for the hard-core model when $\lambda<\frac{1}{2d-1}$, and that this Gibbs measure satisfies exponential strong spatial mixing. The disagreement percolation condition broadens this regime to $\lambda<\frac{p_c(\Z^d)}{1-p_c(\Z^d)}$ (see~\cite{van1994percolation}). It has further been shown by Weitz~\cite{weitz2006counting} that this continues to hold when $\lambda<\lambda_c(2d-1)$, where
\[ \lambda_c(\Delta) := \frac{\Delta^\Delta}{(\Delta-1)^{\Delta+1}} \]
is the critical fugacity for the hard-core model on the regular tree of degree $\Delta+1$.
Moreover, it is shown in~\cite{sinclair2017spatial} that exponential strong spatial mixing further holds when $\lambda<\lambda_c(\mu_d)$, where $\mu_d$ is the connective constant of $\Z^d$ (this improves Weitz's result since $\mu_d<2d-1$ and $\lambda_c(\Delta)$ is decreasing in $\Delta$). We refer the reader to~\cite{sinclair2017spatial}, where the known bounds on the connective constants of several lattices are conveniently summarized in a table.
For $\Z^2$, the authors of~\cite{sinclair2017spatial} also show (by a different means than improving the connective constant of $\Z^2$) that exponential strong spatial mixing holds when $\lambda \le 2.538$.
On the other hand, for $d \ge 2$ and $\lambda > Cd^{-1/3}\log^2 d$, with $C$ a universal constant, it is known that there are multiple Gibbs measures for the hard-core model on $\Z^d$ at fugacity $\lambda$~\cite{peled2014odd}. 
In fact, in this regime, there are two 2-periodic extreme Gibbs measures, $\mu^\text{e}$ and $\mu^\text{o}$, characterized by unequal densities on the even and odd sublattices (when $\lambda > Cd^{-1/4}\log^2 d$, it has been shown that these are the only periodic extreme Gibbs measures~\cite{peled2017condition}), and the translation-invariant Gibbs measure $\tfrac12(\mu^\text{e}+\mu^\text{o})$ is not totally ergodic and hence not \ffiid.
Since the existence of a safe symbol ensures that~\eqref{eq:no-ffiid-assumption} holds, Theorem~\ref{thm:no-ffiid} implies that no Gibbs measure can be \ffiid\ in this regime.
Thus, we obtain the following corollary.

\begin{cor}\label{cor:hard-core}
	For $d \ge 2$ and $\lambda<\lambda_c(\mu_d)$, the unique Gibbs measure for the hard-core model on $\Z^d$ at fugacity $\lambda$ is \ffiid\ with exponential tails. Moreover, in the case $d=2$, this also holds when $\lambda \le 2.538$. In addition, there exists a universal constant $C$ such that for $d \ge 2$ and $\lambda > Cd^{-1/3}\log^2 d$, no Gibbs measure for the hard-core model on $\Z^d$ at fugacity $\lambda$ is \ffiid.
\end{cor}

\begin{remark}\label{rem:hard-core-anti-monotone}
At any fugacity, the hard-core model is anti-monotone (see Section~\ref{sec:discussion} for the definition; see also~\cite{van1994percolation,haggstrom1998exact}).
Using that $\Z^d$ is bipartite and that the hard-core model has only nearest-neighbor interactions, allows to take advantage of the anti-monotonicity is a strong manner.
Define a partial order $\preceq$ on $\{0,1\}^{\Z^d}$ in which $\omega \preceq \omega'$ if and only if $\omega_v \le \omega'_v$ for all even $v$ and $\omega_v \ge \omega'_v$ for all odd $v$ (where $v\in \Z^d$ is said to be \emph{even} or \emph{odd} according to the sum of its coordinates).
Let $\omega^\text{e}$ and $\omega^\text{o}$ be unique minimal and maximal elements of $\{0,1\}^{\Z^d}$ with respect to this partial order and note that $\omega^\text{e},\omega^\text{o} \in \Omega$ (these elements correspond to the independent sets induced by the even and odd sublattices).
It follows (see, e.g., \cite[Lemma~3.2]{van1994percolation}) that $P^{\omega^\text{e}_{\partial V}}_V$ and $P^{\omega^\text{o}_{\partial V}}_V$ converge weakly as $V$ increases to $\Z^d$, giving rise to two Gibbs measures, $\mu^\text{e}$ and $\mu^\text{o}$, which are the smallest and largest Gibbs measures in the sense of stochastic domination (with respect to the partial order $\preceq$), and that there is a unique Gibbs measure if and only if $\mu^\text{e}=\mu^\text{o}$.
Thus, when there is more than one Gibbs measure, Theorem~\ref{thm:no-ffiid} implies that no Gibbs measure is \ffiid. On the other hand, when there is a unique Gibbs measure, the arguments in~\cite{van1999existence} (adapted to anti-monotone systems instead of monotone systems) yield that the unique Gibbs measure is \ffiid\ (see Section~\ref{sec:discussion}). We remark that, since the existence of multiple Gibbs measures in the hard-core model is not known to be monotonic with respect to the fugacity~\cite{brightwell1999nonmonotonic,haggstrom2002monotonicity}, we do not know that there exists a critical fugacity for the property of being \ffiid.
\end{remark}

\subsubsection{{\bf The Widom--Rowlinson model} with $q \ge 2$ types of particles}\label{sec:WR-model}
Set $S:=\{0,1,\dots,q\}$ and let $\Omega$ be the set of configurations $x \in S^{\Z^d}$ satisfying $x_u x_v=0$ or $x_u=x_v$ for any adjacent $u$ and~$v$.
Given a parameter $\lambda>0$, called the \emph{activity}, the specification is defined by
\[ P^\tau_V(\omega) := \frac{\1_{\Omega^\tau_V}(\omega)}{Z^\tau_V} \cdot \lambda^{\sum_{v \in V} \1_{\{\omega_v \neq 0 \}}} .\]
The high noise condition implies that there is a unique Gibbs measure when $q\lambda<\frac{1}{2d-1}$, and that this measure satisfies exponential strong spatial mixing.
Dobrushin's unique condition implies the same whenever $q\lambda<\min\{\frac{q}{2d-1},\frac{1}{2d-1-2d/q}\}$, and the disagreement percolation condition implies this whenever $q\lambda < \frac{p_c(\Z^d)}{1-p_c(\Z^d)}$. Thus, in these regimes, Theorem~\ref{thm:main} implies that the unique Gibbs measure is \ffiid\ with exponential tails. On the other hand, when $\lambda$ is sufficiently large as a function of the dimension ($\lambda > Cd^{-1/8}\log d$ suffices~\cite{peled2017condition}), it has been shown that there are multiple ergodic Gibbs measures~\cite{lebowitz1971phase,runnels1974phase,burton1995new,peled2017condition}, and since the existence of a safe symbol implies that~\eqref{eq:no-ffiid-assumption} holds for any two such measures (see Section~\ref{sec:discussion}), Theorem~\ref{thm:no-ffiid} implies that these Gibbs measures are not \ffiid. We remark that, in the special case of $q=2$, the model is monotone, so that whenever there is a unique Gibbs measure it is \ffiid\ (see Section~\ref{sec:discussion}).

\subsubsection{\bf The beach model}\label{sec:beach-model}
Set $S:=\{-2,-1,1,2\}$ and let $\Omega$ be the set of configurations $x \in S^{\Z^d}$ satisfying that $x_u x_v \ge -1$ for any adjacent $u$ and~$v$.
Given a parameter $\lambda>0$, called the \emph{activity}, the specification is defined by
\begin{equation}\label{eq:beach-spec}
P^\tau_V(\omega) := \frac{\1_{\Omega^\tau_V}(\omega)}{Z^\tau_V} \cdot \lambda^{\sum_{v \in V} |\omega_v|} .
\end{equation}
H{\"a}ggstr{\"o}m~\cite{haggstrom1996phase} showed that there exists a critical activity $\lambda_c(d) \in (0,\infty)$ such that there is a unique Gibbs measure for the beach model on $\Z^d$ at activity $\lambda<\lambda_c(d)$ and multiple ergodic Gibbs measures at activity $\lambda>\lambda_c(d)$. Moreover, there exists a universal constant $C$ such that
\[ \frac{2}{2d^2 + d - 1} \le \lambda_c(d) \le 1 + C d^{-1/4} \log^2 d .\]
The lower bound is shown in~\cite{haggstrom1996phase} and the upper bound follows from the results in~\cite{peled2017condition} (the upper bound $\lambda_c(d) \le 4e\cdot 28^d$ was first shown by Burton and Steif~\cite{burton1994non}).
As we show below, no Gibbs measure is \ffiid\ when $\lambda>\lambda_c(d)$.
On the other hand, showing that the unique Gibbs measure is \ffiid\ when $\lambda<\lambda_c(d)$ is not straightforward, even for small $\lambda$. This is because neither Dobrushin's uniqueness condition nor disagreement percolation can be directly applied to this model, as their assumptions do not hold for any $\lambda>0$. Nevertheless, we can show the following.

\begin{cor}\label{cor:beach-model}
	Let $d \ge 2$ and let $\lambda_c(d)$ be the critical activity for the beach model on $\Z^d$.
	Let $\lambda>0$ and let $X$ be a Gibbs measure for the beach model on $\Z^d$ at activity $\lambda$.
	If $\lambda<2/(2d^2+d-1)$ then $X$ is \ffiid\ with exponential tails, if $\lambda<\lambda_c(d)$ then $X$ is \ffiid, and if $\lambda>\lambda_c(d)$ then $X$ is not \ffiid.
\end{cor}

\begin{proof}
	The proof of the first part of the corollary is based on a correspondence between the beach model and the so-called site-centered ferromagnetic, which is an Ising-type model (with finite-range ferromagnetic interactions on more than just pairs of sites). We briefly describe some properties of this correspondence; see~\cite[Section~4]{haggstrom1996phase} for details and proofs (though the model is defined in~\cite{haggstrom1996phase} only for rational $\lambda$, the proofs there extend to any positive $\lambda$).
	Define a $\pm 1$-valued random field $Z$ by $Z_v := \text{sign}(X_v)$ so that $X = (Z_v |X_v|)_{v \in \Z^d}$. Then $Z$ is a Gibbs measure for the site-centered ferromagnet (one may take the induced specification of $Z$ as the definition of the site-centered ferromagnet).
	 Moreover, conditioned on $Z$, the random field $(|X_v|)_{v \in \Z^d}$ is an \iid\ process, where, for every $v$, we have that $|X_v|=1$ almost surely if $Z_u \neq Z_v$ for some $u \in \partial v$, and $|X_v|-1$ is a Bernoulli random variable with parameter $\lambda/(1+\lambda)$ otherwise.

When $\lambda < 2/(2d^2+d-1)$, it is shown in~\cite{haggstrom1996phase} (see the proof of Lemma~4.9 there) that $Z$ satisfies Dobrushin's uniqueness condition. Thus, $Z$ satisfies exponential strong spatial mixing so that, by Theorem~\ref{thm:main}, it is \ffiid\ with exponential tails (actually, $Z$ is a 2-Markov random field, but Theorem~\ref{thm:main} extends to this case; see Remark~\ref{rem:r-Markov}).
Let $W=(W_v)_{v \in \Z^d}$ be an \iid\ process, independent of $Z$, with $W_v$ distributed $\text{Ber}(\lambda/(1+\lambda))$.
Since $X$ is a finitary factor of $(Z,W)$ with (deterministic) coding radius~1, it easily follows that $X$ is \ffiid\ with exponential tails.

When $\lambda<\lambda_c(d)$, there is a unique Gibbs measure.
Since the beach model is monotone, it follows that $X$ is \ffiid\ (see Section~\ref{sec:discussion}).

When $\lambda > \lambda_c(d)$, there are two distinct ergodic Gibbs measures (characterized by an equal density of positive and negative values)~\cite{haggstrom1996phase,peled2017condition}. Thus, Theorem~\ref{thm:no-ffiid} will imply that $X$ is not \ffiid\ once we show that~\eqref{eq:no-ffiid-assumption} holds for any two ergodic Gibbs measures. In fact, we shall show that the stronger~\eqref{eq:soft-constraint-lower-bound} holds.
Specifically, we show that for any finite $U \subset \Z^d$ and any $\tau \in \Omega_U$,
\[ \Pr(X_U=\tau) \ge c^{(2d+1)^2|U|} \qquad\text{where }c := \min_{\tau \in \Omega_{\partial v}} \min_{s \in \Omega^\tau_v} P^\tau_v(s) = \frac{\min\{1,\lambda\}}{2+\lambda} .\]
In fact, this follows from~\cite[Proposition~4.9]{briceno2017topological} (perhaps with a different constant $c$), since $\Omega$ satisfies a property called topological strong spatial mixing. Nevertheless, we provide a short self-contained proof.
Denote $W := U \cup \partial U$ and $V := W \cup \partial W$. It suffices to show that, for any $\xi \in \Omega_{\partial V}$,
\[ \Pr(X_U=\tau \mid X_{\partial V}=\xi) \ge c^{|V|} \ge c^{(2d+1)^2|U|} .\]
This will follow if we show that there exists $\omega \in \Omega^\xi_V$ such that $\omega_U=\tau$, and that, for any $\omega \in \Omega^\xi_V$, $P^\xi_V(\omega) \ge c^{|V|}$. The latter easily follows from~\eqref{eq:beach-spec}, since $Z^\tau_V \le (\lambda(2+\lambda))^{|V|}$.
To see the former, first observe that, for any $\eta \in \Omega_{\partial v}$, we have $1 \in \Omega^\eta_v$ (we slightly abuse notation here by identifying $\Omega^\eta_v$ with a subset of $S$) if and only if $\eta_u \neq -2$ for all $u \in \partial v$. A similar property holds for $-1$. Thus, there exists $\tau' \in \Omega_W$ such that $\tau'_U=\tau$ and $\tau'_v \in \{-1,1\}$ for all $v \in \partial U$, and similarly, there exists $\xi' \in \Omega_{\partial V \cup \partial W}$ such that $\xi'_{\partial V}=\xi$ and $\xi'_v \in \{-1,1\}$ for all $v \in \partial W$. It follows that the concatenation of $\tau'$ and $\xi'$ is a feasible configuration on $V \cup \partial V$. That is, the element $\omega \in S^V$ defined by $\omega_W := \tau'$ and $\omega_{\partial W} = \xi'_{\partial W}$ belongs to $\Omega^\xi_V$.
\end{proof}

\begin{remark}
	In light of the correspondence established in~\cite{haggstrom1996phase} between the beach model and the site-centered ferromagnet (which is monotone and enjoys an additional monotonicity property with respect to $\lambda$), it seems plausible that one may be able to obtain a sharp transition for the existence of a finitary coding (in a similar manner as for the Ising model). Namely, that for $\lambda<\lambda_c(d)$, the unique Gibbs measure is \ffiid\ with exponential tails. We do not pursue this here.
\end{remark}

\subsubsection{\bf The monomer-dimer model}
This model consists of a random matching in $\Z^d$, i.e., a collection of disjoint edges (also called dimers), in which a matching $\omega$ in finite-volume is chosen with probability proportional to $\lambda$ to the power of the number of dimers. In particular, configurations in this model are elements of $\{0,1\}^{E(\Z^d)}$ (in fact, this model is precisely the hard-core model on the line graph of $\Z^d$). Strictly speaking, such models do not fall into the setting of our theorems, but the results extend to this case (see Remark~\ref{rem:other-graphs}) (alternatively, we could instead regard configurations as elements of, say, $\{0,1,\dots,2d\}^{\Z^d}$ by encoding the direction of the dimer incident to a vertex, where 0 means no dimer is present). Van den Berg~\cite{van1999absence} showed that this model satisfies exponential strong spatial mixing for any finite value of $\lambda>0$. It follows that the monomer-dimer model on $\Z^d$ is \ffiid\ with exponential tails for any $\lambda$.

\subsubsection{\bf Graph homomorphisms}\label{sec:homomorphism-example}
In~\cite{brightwell2000gibbs,briceno2017strong}, it is shown that a large class of nearest-neighbor spin systems satisfy exponential weak spatial mixing. Theorem~\ref{thm:ffiid} shows that the unique Gibbs measure in such a system is \ffiid.
We do not describe the general class of these systems (they are certain weighted homomorphisms to dismantlable graphs), but instead give one simple example in which weak spatial mixing is satisfied, but strong spatial mixing is not.
Let $\Omega$ be the set of all configurations $x \in \{0,1,2,3\}^{\Z^d}$ such that $\{x_u,x_v\} \notin \{ \{0,2\},\{1,3\},\{2,2\} \}$ for any adjacent $u$ and $v$, and let $P^\tau_V$ be the probability measure on $\Omega^\tau_V$ in which $P^\tau_V(\omega)$ is proportional to $\lambda^{\sum_{v \in V} \1_{\{\omega_v=0\}}}$. It is shown in~\cite[Proposition~9.2]{briceno2017strong} that this specification satisfies exponential weak spatial mixing when $\lambda$ is sufficiently large (as a function of $d$), but does not satisfy strong spatial mixing (with any rate function) for any $\lambda$.
Thus, for any $d \ge 2$, when $\lambda$ is sufficiently large, the unique Gibbs measure is \ffiid, but we do not know whether it is \ffiid\ with exponential tails.
On the other hand, using the results in~\cite{peled2017condition}, one may check that, for any fixed $\lambda \neq 1$, when $d$ is sufficiently large, no Gibbs measure is \ffiid; for $\lambda<1$, the reason being that there is a unique translation-invariant Gibbs measure and this measure is not totally ergodic (it is a mixture of two 2-periodic extreme Gibbs measures, characterized by having an unequal density of $\{0,2\}$ and $\{1,3\}$ on the even and odd sublattices), and for $\lambda>1$, the reason being that there are precisely two ergodic Gibbs measures (characterized by having a very high density of either $\{0,1\}$ or $\{0,3\}$) so that Theorem~\ref{thm:no-ffiid} applies (although $\Omega$ does not have a safe symbol and does not satisfy topological strong spatial mixing, one may check in a similar manner as in the proof of Corollary~\ref{cor:beach-model} that~\eqref{eq:soft-constraint-lower-bound} holds, so that assumption~\eqref{eq:no-ffiid-assumption} of the theorem is satisfied).

\subsection{Discussion}
\label{sec:discussion}

This work deals with a relation between finitary codings and spatial mixing properties. We give a short survey on each of these.

Finitary codings and finitary isomorphisms (invertible finitary codings, whose inverses are also finitary) have a long history (see, e.g., \cite{serafin2006finitary} for a survey), beginning with Meshalkin's example of finitarily isomorphic \iid\ processes~\cite{meshalkin1959case}, and culminating in the work of Keane and Smorodinsky~\cite{keane1979bernoulli} who showed that \iid\ processes of the same entropy are finitarily isomorphic, strengthening the celebrated Ornstein isomorphism theorem~\cite{ornstein1970bernoulli}.
The theory of finitary codings and finitary isomorphisms for one-dimensional processes is rather well established~\cite{keane1977class,keane1979bernoulli,keane1979finitary,akcoglu1979finitary,rudolph1981characterization,rudolph1982mixing,krieger1983finitary,fiebig1984return,schmidt1984invariants,del1990bernoulli,smorodinsky1992finitary,harvey2006universal,harvey2011invariant}.
Finitary codings of random fields in one or more dimensions have also been studied both in particular models such as the Ising model~\cite{van1999existence,spinka2018finitary}, proper colorings~\cite{timar2011invariant,angel2012stationary,holroyd2017one,holroyd2017finitary,holroyd2017mallows}, random walk in random scenery~\cite{steif2001thet,keane2003finitary,den2006random} and perfect matchings~\cite{soo2009translation,timar2009invariant,lyons2011perfect}, and in general settings~\cite{junco1980finitary,van1999existence,haggstrom2000propp,harvey2006universal,galves2008perfect,spinka2018finitary}.
While any finite-state mixing Markov chain in one dimension is isomorphic to an \iid\ process (and is a finitary factor of any \iid\ process with larger entropy~\cite{harvey2006universal}), this is not the case for Markov random fields in two or more dimensions~\cite{slawny1981ergodic}.
Moreover, there exists a Markov random field which is isomorphic to an \iid\ process, but which is not finitarily isomorphic to such a process. On the one hand, del Junco showed that there always exists a finitary coding from any ergodic Markov random field to any finite-valued \iid\ process with lower entropy~\cite{junco1980finitary}. On the other hand, however, there does exist a Markov random field which is a factor of an \iid\ process (and is thus isomorphic to an \iid\ process~\cite{katznelson1972commuting}), but which is not \ffiid. The plus-state of the Ising model above the critical inverse temperature is such an example~\cite{van1999existence}.

Spatial mixing properties have been of interest, among other reasons, because of their implications on uniqueness of Gibbs measures and on efficient approximation of counting problems. 
They have been studied both in particular models such as the Potts model~\cite{goldberg2006improved,mossel2008rapid,lubetzky2013cutoff,yin2015spatial}, proper colorings~\cite{goldberg2005strong,achlioptas2005rapid,jalsenius2009strong,gamarnik2015strong}, the hard-core (and monomer-dimer) model~\cite{van1999absence,weitz2006counting,sinclair2017spatial}, and in general settings~\cite{martinelli1994approach,alexander1998weak,brightwell2000gibbs,alexander2004mixing,dyer2004mixing,nair2007correlation,li2013correlation,marcus2013computing,lubetzky2014cutoff,briceno2017strong,blanca2018spatial}.
In two-dimensional models with no hard constraints, exponential weak spatial mixing implies exponential strong spatial mixing for squares~\cite{martinelli19942} (we elaborate on this below).
Exponential strong (and in some cases also weak) spatial mixing is intimately related to optimal temporal mixing of natural local dynamics (such as Glauber dynamics and various block dynamics)~\cite{stroock1992logarithmic,martinelli1994approach,martinelli1994approach2,martinelli1999lectures,dyer2004mixing,weitz2004mixing} and has even been used to establish cutoff, a sharp transition in the convergence to equilibrium~\cite{lubetzky2013cutoff,lubetzky2014cutoff}.

Rapid temporal mixing of local dynamics is closely related to perfect sampling (also called exact simulation), a subject which has received much attention (see, e.g., \cite{propp1998coupling,dimakos2001guide}).
A simple and useful method, called coupling-from-the-past, was devised by Propp and Wilson~\cite{propp1996exact} for perfectly sampling from the stationary distribution of a one-dimensional Markov chain on a finite state-space (e.g., a stochastic spin system on a finite graph).
This method works particularly well for monotone spin systems and even for anti-monotone ones~\cite{haggstrom1998exact}.
In the absence of monotonicity, auxiliary processes (these have been given several names, including bounding chains, dominating processes, super-PCA and envelope-PCA) have been used together with coupling-from-the-past to provide efficient perfect sampling algorithms~\cite{huber1998exact,haggstrom1999exact,
haggstrom2000propp,kendall2000perfect,
huber2004perfect,buvsic2013probabilistic}.
An extension of coupling-from-the-past to Markov random fields (nearest-neighbor spin systems on infinite graphs) was a basic ingredient underlying the constructions of finitary codings in~\cite{van1999existence,haggstrom2000propp,spinka2018finitary}, and is such in the present paper as well. In particular, these constructions also provide an algorithm for perfectly sampling the marginal $X_V$ of a Markov random field $X$ (for which the results apply) on any finite set $V \subset \Z^d$. This is somewhat surprising since, in such random fields, the value $X_v$ at any given site $v$ is usually correlated with the value at infinitely many other sites.
We mention that perfect sampling of non-Markov random fields (having infinite-range interactions) has also been studied~\cite{de2008exact,galves2008perfect,galves2010perfect,de2012developments}.

For (discrete nearest-neighbor) monotone spin systems, van den Berg and Steif~\cite{van1999existence} established Theorem~\ref{thm:main} under weaker assumptions (they did not formulate the results in the precise manner that we now describe, but these follow from their results).
Let us first explain what we mean by a monotone system.
A specification $P$ is said to be \emph{monotone} if there exists a linear ordering $\le$ on $S$ such that, for any vertex $v \in \Z^d$ and any boundary conditions $\tau,\tau' \in \Omega_{\partial v}$ such that $\tau \le \tau'$ pointwise (i.e., $\tau_u \le \tau'_u$ for all $u \in \partial v$), we have that $P^\tau_v$ is stochastically dominated by $P^{\tau'}_v$, i.e.,
\[ P^\tau_v(\omega_v \ge s) \le P^{\tau'}_v(\omega_v \ge s) \qquad\text{for all }s \in S .\]
Many well-known spin systems are monotone, including the Ising model (see Section~\ref{sec:Potts-model}), the Widom-Rowlinson model with two types of particles (see Section~\ref{sec:WR-model}) and the beach model (see Section~\ref{sec:beach-model}). 
Let us also mention a similar anti-monotonic property. A specification $P$ is said to be \emph{anti-monotone} if there exists a linear ordering $\le$ on $S$ such that $P^\tau_v$ is stochastically dominated by $P^{\tau'}_v$ whenever $\tau \ge \tau'$ pointwise. The hard-core model (see Section~\ref{sec:hard-core-model}) is an example of an anti-monotone spin system.
Van den Berg and Steif showed that when there is a unique Markov random field with a given monotone specification, this random field is \ffiid. Furthermore, using a result of Martinelli and Olivieri~\cite{martinelli1994approach}, they showed that, for monotone systems with no hard constraints, exponential weak spatial mixing implies \ffiid\ with exponential tails (they also showed that this implies \fvffiid, i.e., finitary factor of a finite-valued \iid\ process, which was improved to \fvffiid\ with stretch-exponential tails in~\cite{spinka2018finitary}).
An inspection of the proofs in~\cite{martinelli1994approach,van1999existence} reveals that these results carry over to anti-monotone systems (for this it is important that $\Z^d$ is bipartite and that the spin system has only nearest-neighbor interactions, rather than any finite-range interaction; see~\cite{haggstrom1998exact}).
Thus, for monotone and anti-monotone systems with no hard constraints, Theorem~\ref{thm:main} yields weaker results than those in~\cite{van1999existence}.
On the other hand, when exponential strong spatial mixing is satisfied, though Theorem~\ref{thm:main} does not yield a new result, we provide a self-contained proof (whereas the result in~\cite{van1999existence} appeals to~\cite{martinelli1994approach}).
Also, when~\eqref{eq:positive-coupling-prob} holds (e.g., for the critical two-dimensional Ising model), the result in~\cite{van1999existence} gives no information on the coding radius in general (though specifically for the Ising model, one could get such information by appealing to results about the mixing time~\cite{lubetzky2012critical}), whereas our proof yields a more explicit power-law bound on its tails.

H{\"a}ggstr{\"o}m and Steif~\cite{haggstrom2000propp} showed how one could drop the monotonicity assumption in exchange for a \emph{high noise} condition, which is, with our definitions,
\begin{equation}\label{eq:high-noise}
\gamma(\{v\},\{v\}) > 1-\frac{1}{2d} .
\end{equation}
The quantity $\gamma(\{v\},\{v\})$ is called the \emph{multigamma admissibility}, and, as mentioned, is the probability under an optimal coupling of $(P^\tau_v)_{\tau \in \Omega_{\partial v}}$ that the value at $v$ is the same for all boundary conditions (see Section~\ref{sec:couplings}).
The authors of~\cite{haggstrom2000propp} proved that any translation-invariant Markov random field satisfying~\eqref{eq:high-noise} is \ffiid\ with exponential tails (as in the case above, they also show \fvffiid\ and this was improved to \fvffiid\ with stretched-exponential tails in~\cite{spinka2018finitary}).

Let us further compare the high noise condition~\eqref{eq:high-noise} to two well-known weaker conditions existing in the literature, which are sufficient for exponential strong spatial mixing. These are Dobrushin's uniqueness condition~\cite{Dobrushin1968TheDe} and van den Berg and Maes's disagreement percolation condition~\cite{van1994disagreement} (we formulate these for translation-invariant Markov random fields, though both are more general).
As will immediately be evident, the main advantage of these conditions over the above high noise condition is that they involve comparisons between pairs of boundary conditions, whereas~\eqref{eq:high-noise} involves a simultaneous comparison between all boundary conditions. In particular, it is the case that each of the two conditions is implied by~\eqref{eq:high-noise}.
Recall that $\Sigma_{\tau,\tau'}$ is the set on which $\tau$ and $\tau'$ disagree.
The uniqueness condition of Dobrushin is
\begin{equation}\label{eq:dobrushin-uniqueness}
 \sum_{u \in \partial v} ~ \max_{\substack{\tau,\tau' \in \Omega_{\partial v}\\\Sigma_{\tau,\tau'}=\{u\}}} \|P^\tau_v - P^{\tau'}_v\| < 1 .
\end{equation}
Since $\|P^\tau_v - P^{\tau'}_v\| \le 1-\gamma(\{v\},\{v\})$ for any $\tau,\tau' \in \Omega_{\partial v}$, it is immediate that~\eqref{eq:high-noise} implies~\eqref{eq:dobrushin-uniqueness}.
It is also known that~\eqref{eq:dobrushin-uniqueness} implies exponential strong spatial mixing (see~\cite[Theorem~3.2]{weitz2005combinatorial} and~\cite[Theorem~3.3, Corollary~3.12, Theorem~4.6]{weitz2004mixing}).
One may interpret~\eqref{eq:dobrushin-uniqueness} as a bound on the total influence (of the neighbors a site) on a site.
Similar conditions have been studied~\cite{dobrushin1985constructive,dobrushin1985completely,stroock1992logarithmic,weitz2005combinatorial,weitz2004mixing}, allowing for other metrics, larger finite-size criteria (rather than a single-site criterion), and total influence \emph{of} a site (dual to the notion of total influence \emph{on} a site).
The disagreement percolation condition of van den Berg and Maes is, in its simplest form,
\begin{equation}\label{eq:disagreement-percolation}
 \|P^\tau_v - P^{\tau'}_v\| < p_c(\Z^d) \qquad\text{for all }\tau,\tau' \in \Omega_{\partial v} ,
\end{equation}
where $p_c(\Z^d)$ is the critical probability for site-percolation on $\Z^d$. Using that $p_c(\Z^d)>1/(2d-1)$ for $d \ge 2$, one sees that~\eqref{eq:high-noise} implies~\eqref{eq:disagreement-percolation}. Moreover, the proof in~\cite{van1994disagreement} shows that \eqref{eq:disagreement-percolation} implies exponential strong spatial mixing.
Neither~\eqref{eq:dobrushin-uniqueness} nor~\eqref{eq:disagreement-percolation} follow from one another, and there are situations in which either one provides better bounds than the other (see~\cite{van1994disagreement,georgii2001random} for a discussion and comparison between the two conditions).
As both~\eqref{eq:dobrushin-uniqueness} and~\eqref{eq:disagreement-percolation} imply exponential strong spatial mixing, by Theorem~\ref{thm:ffiid}, they are both sufficient conditions for a translation-invariant Markov random field to be \ffiid\ with exponential tails.

Strong spatial mixing, by definition, allows to couple any two measures $P^\tau_V$ and $P^{\tau'}_V$ so that, with high probability, they agree on the region of $V$ which is far away from the set $\Sigma_{\tau,\tau'}$ of disagreement between $\tau$ and $\tau'$. Such a coupling is ``optimal'' for the pair of boundary conditions $\{\tau,\tau'\}$.
The main challenge in the proof of part~\eqref{thm:main-b} of Theorem~\ref{thm:main} is to construct a coupling of the measures $(P^\tau_V)_\tau$, for all possible boundary conditions $\tau \in \Omega_{\partial V}$, which is simultaneously nearly optimal for any pair of boundary conditions (in fact, we need this for more than just pairs of boundary conditions). It is not immediately clear that this is at all possible (see Remark~\ref{rem:optimal-couplings} in Section~\ref{sec:couplings}; this is somewhat analogous to the difference between stochastic monotonicity and realizable monotonicity~\cite{fill2001stochastic}).
Although versions of the proof of Theorem~\ref{thm:contraction-implies-ffiid} have appeared before~\cite{haggstrom2000propp,spinka2018finitary} (the proof is based on a contracting property for an auxiliary ``bounding'' random field~\cite{haggstrom1999exact,diaconis1999iterated,haggstrom2000propp,huber2004perfect,spinka2018finitary}), the key ingredient, Theorem~\ref{thm:grand-coupling-exists}, which establishes the existence of such a coupling, seems to be new.

As mentioned above, in two-dimensional models with no hard constraints (i.e., when the topological support $\Omega$ is $S^{\Z^2}$), exponential weak spatial mixing implies exponential strong spatial mixing for squares~\cite{martinelli19942} (see also~\cite{alexander2004mixing}). In fact, for this conclusion, it even suffices to assume exponential weak spatial mixing only for subsets of a sufficiently large square (see the remark just after Theorem~1.1 in~\cite{martinelli19942}).
More precisely, for any exponentially decaying rate function $\rho$, there exists $L_0$ and an exponentially decaying rate function $\rho'$ for which, if~\eqref{eq:WSM-def} holds for all $(U,V)$ such that $V \subset \Lambda_{L_0}$, then~\eqref{eq:SSM-def} holds (with rate function $\rho'$) for all $(U,V)$ such that $V=\Lambda_L$ for some $L$.
An inspection of the proof of part~\eqref{thm:main-b} of Theorem~\ref{thm:main} for the two-dimensional case, shows that we only require exponential strong spatial mixing for squares (see Remark~\ref{rem:two-dimensional-proof}). Thus, we obtain that a translation-invariant Markov random field on $\Z^2$ that has no hard constraints and satisfies exponential weak spatial mixing (even just for subsets of a sufficiently large square) is \ffiid\ with exponential tails.

As we have also mentioned, van den Berg and Steif~\cite{van1999existence} showed that a phase transition presents an obstruction for being \ffiid\ in models with no hard constraints (their proof is based on the so-called blowing-up property, which was established for finitary factors of \iid\ processes by Marton and Shields~\cite{marton1994}). One explanation behind this phenomenon is that the ergodic theorem holds at an exponential rate for \iid\ processes, and \ffiid\ preserves this property, while a phase transition prohibits an exponential rate in the ergodic theorem, as the configuration may ``jump'' from one phase to another at a sub-exponential ``cost'' (such a transition comes with a boundary effect, but no cost in the bulk) (see~\cite{van1999existence} for a more detailed discussion).
Theorem~\ref{thm:no-ffiid} generalizes this phenomenon to models with hard constraints (the idea of the proof is based on the explanation above and on some remarks in~\cite{van1999existence}). Indeed, it is easy check that~\eqref{eq:no-ffiid-assumption} holds in the absence of hard constraints, since, in this case, $P^\tau_v(s)>0$ for any $\tau \in \Omega_{\partial v}$ and $s \in S$, implying the stronger property that, for some $c>0$,
\begin{equation}\label{eq:soft-constraint-lower-bound}
\Pr(X_{\partial \Lambda_n}=\tau) \ge e^{-cn^{d-1}} \qquad\text{for any }n \ge 1\text{ and any }\tau \in \Omega_{\partial \Lambda_n} .
\end{equation}
One may similarly check that~\eqref{eq:soft-constraint-lower-bound} holds whenever there is a so-called ``safe symbol'', that is, a value $s_0 \in S$ for which $P^\tau_v(s_0)>0$ for any $\tau \in \Omega_{\partial v}$. In fact, \eqref{eq:soft-constraint-lower-bound} holds in much greater generality, namely, whenever the topological support $\Omega$ satisfies topological strong spatial mixing~\cite[Proposition~4.9]{briceno2017topological} (see Section~\ref{sec:beach-model} for a hands-on proof in the case of the beach model). Moreover, \eqref{eq:soft-constraint-lower-bound} also holds in other natural situations in which there is neither a safe symbol nor topological strong spatial mixing (e.g., the example given in Section~\ref{sec:homomorphism-example}). Thus, for such models, Theorem~\ref{thm:no-ffiid} says that a phase transition (i.e., the existence of more than one ergodic Markov random field with the given specification) is an obstruction for being \ffiid.

\medbreak

We conclude with some open problems:
\begin{enumerate}
	\item Does exponential weak spatial mixing imply \ffiid\ with finite expected coding radius? (our proof gives a coding radius having heavy power-law tails)
	\item Does exponential weak spatial mixing imply \fvffiid? (we showed that it implies \ffiid)
	\item Does weak spatial mixing with \emph{some} rate function imply \ffiid?
	\item Does exponential strong spatial mixing imply \fvffiid\ with exponential tails? (we showed that it implies \ffiid\ with exponential tails and also \fvffiid\ with stretched-exponential tails)
	\item Does exponential strong spatial mixing imply the existence of a finitary coding from every finite-valued \iid\ process with larger entropy?
	\item Does \ffiid\ with exponential tails imply (exponential) weak spatial mixing?
	\item Can assumption~\eqref{eq:no-ffiid-assumption} be removed from Theorem~\ref{thm:no-ffiid} (instead assuming only that $X$ and $X'$ have the same topological support)?
\end{enumerate}

\subsection{Notation}
We consider $\Z^d$ as a graph in which two vertices $u$ and $v$ are adjacent if $|u-v|=1$, where $|v|=\|v\|_1:=|v_1|+\cdots+|v_d|$ denotes the $\ell_1$-norm.
For sets $U,V \subset \Z^d$, we write $\dist(U,V) := \min_{u \in U,v \in V} |u-v|$ and $\dist(u,V) := \dist(\{u\},V)$. We denote $\partial V := \{ u \in \Z^d : \dist(u,V)=1 \}$ and $\partial v := \partial \{v\}$.
We use $\zero$ to denote the origin $(0,\dots,0) \in \Z^d$.

\subsection{Organization}
In Section~\ref{sec:technique}, we provide the basic technique of proof used to construct the finitary codings. We then apply this in Section~\ref{sec:weak-mixing} to prove Theorem~\ref{thm:ffiid} and part~\eqref{thm:main-a} of Theorem~\ref{thm:main}. In Section~\ref{sec:strong-mixing}, we again apply this technique in order to prove part~\eqref{thm:main-b} of Theorem~\ref{thm:main}. Finally, in Section~\ref{sec:no-ffiid}, we prove Theorem~\ref{thm:no-ffiid}.

\subsection{Acknowledgments}
I would like to thank Raimundo Brice\~no, Nishant Chandgotia, Hugo Duminil-Copin and Jeff Steif for useful discussions. I am especially grateful to Matan Harel for many discussions throughout the preparation of this paper.

\section{Proof technique and overview}
\label{sec:technique}

There is a common thread lying at the base of the proofs of the two parts of Theorem~\ref{thm:main} (and also Theorem~\ref{thm:ffiid}), which is the use of block dynamics in conjunction with the method of coupling-from-the-past. We begin by explaining the dynamics in Section~\ref{sec:dynamics}. We then discuss couplings in Section~\ref{sec:couplings}, an important ingredient for the implementation of coupling-from-the-past. Next, in Section~\ref{sec:coupled-dynamics}, we give a formal definition of the dynamics, coupled from all starting configurations. Finally, we show in Section~\ref{sec:cftp} how these coupled dynamics can be used together with coupling-from-the-past to provide a recipe for constructing finitary codings, and we indicate in Section~\ref{sec:outline} how this method is later applied in Section~\ref{sec:weak-mixing} and Section~\ref{sec:strong-mixing} to obtain Theorem~\ref{thm:main} and Theorem~\ref{thm:ffiid}.

\subsection{Dynamics}
\label{sec:dynamics}
Our proofs rely on a dynamics called the \emph{heat-bath block dynamics}. On a finite graph, this is the following procedure:
\begin{itemize}
	\item Choose a block $\Delta$ (a set of vertices) at random according to some distribution.
	\item Resample the configuration on $\Delta$ according to the conditional distribution given by the current configuration on $\partial \Delta$.
\end{itemize}
On an infinite graph, we may follow a similar procedure, as long as we restrict ourselves to blocks $\Delta$ which have no infinite connected components. As we are considering Markov random fields, this restriction ensures that the resampling step is well-defined by the specification (the configuration is resampled independently in each component).

As we aim to construct a coding, our block distributions cannot be arbitrary and should satisfy certain properties such as translation invariance. We thus restrict ourselves specialized block distributions which arise from the following procedure:
\begin{itemize}
	\item Let $\Delta_0 \subset \Z^d$ be finite and contain $\zero$ -- this is the ``base'' block.
	\item Declare each vertex of $\Z^d$, independently, to be ``active'' with probability $p_0$.
	\item Declare each active vertex to be ``chosen'' if there are no other active vertices within distance $\diam(\Delta_0)+1$ from it.
	\item Let $\Delta$ be the union of the translates of $\Delta_0$ by the chosen vertices.
\end{itemize}
Note that the definition of ``chosen'' ensures that every connected component of $\Delta$ is almost surely a translate of $\Delta_0$.
It is instructive to note that the block distribution obtained is translation-invariant and is \ffiid\ (in fact, \fvffiid\ with bounded coding radius).

\subsection{Couplings}
\label{sec:couplings}

Let $\mu$ and $\nu$ be two probability measures on a common space $\cA$, which we assume to be finite for simplicity. A \emph{coupling} of $\mu$ and $\nu$ is a probability measure $\pi$ on $\cA^2$ which has marginals $\mu$ and $\nu$, i.e., $\pi(A \times \cA)=\mu(A)$ and $\pi(\cA \times A)=\nu(A)$ for any $A \subset \cA$.
The \emph{total-variation distance} between $\mu$ and $\nu$ is
\[ \|\mu-\nu\| := \max_{A \subset \cA} |\mu(A) - \nu(A)| = \tfrac12 \sum_{a \in \cA} |\mu(a)-\nu(a)| .\]
It is well-known that $\|\mu-\nu\|$ equals the probability that a sample from $\mu$ and~$\nu$ are different under an optimal coupling. That is, if $(Z,Z')$ is sampled from a coupling of $\mu$ and $\nu$, then
\[ \Pr(Z=Z') \le 1-\| \mu - \nu \| ,\]
and there exists a coupling for which equality holds. Noting that
\[ 1-\|\mu-\nu\|=\sum_{a \in \cA} \min\{\mu(a),\nu(a)\} ,\]
the above immediately generalizes to any number of distributions $\mu_1,\dots,\mu_k$ on $\cA$. Namely, if $(Z_1,\dots,Z_k)$ is sampled from a coupling of $\mu_1,\dots,\mu_k$, then
\begin{equation}\label{eq:def-optimal-coupling}
\Pr(Z_1=Z_2=\dots=Z_k) \le \sum_{a \in \cA} \min\{\mu_1(a),\dots,\mu_k(a)\} ,
\end{equation}
and there exists a coupling for which equality holds. A coupling for which there is equality in~\eqref{eq:def-optimal-coupling} is called an \emph{optimal coupling} of $\mu_1,\dots,\mu_k$.
When $k \ge 3$, we often refer to a coupling of $\mu_1,\dots,\mu_k$ as a \emph{grand coupling} in order to stress the fact that it is a coupling of multiple measures, and not just of two measures.

We shall be interested in grand couplings of $(P^\tau_V)_{\tau \in \Omega_{\partial V}}$ or $(P^\tau_{V,U})_{\tau \in \Omega_{\partial V}}$ for various finite sets $U \subset V \subset \Z^d$.
Note that such a grand coupling simultaneously couples samples of the Markov random field (on $V$ or $U$) under \emph{all} boundary conditions, and not just under two boundary conditions.
Recall the definition of $\gamma(V,U)$ from~\eqref{eq:def-gamma} and observe that a grand coupling $\pi$ of $(P^\tau_{V,U})_{\tau \in \Omega_{\partial V}}$ is an optimal coupling if and only if
\[ \pi\big(\omega^\tau = \omega^{\tau'}\text{ for all }\tau,\tau' \in \Omega_{\partial V}\big) = \gamma(V,U) .\]
Since any such coupling can be extended to a coupling of $(P^\tau_V)_{\tau \in \Omega_{\partial V}}$, it follows that there is also a grand coupling $\tilde{\pi}$ of $(P^\tau_V)_{\tau \in \Omega_{\partial V}}$ for which
\begin{equation}\label{eq:optimal-coupling}
\tilde\pi\big(\omega^\tau_U = \omega^{\tau'}_U\text{ for all }\tau,\tau' \in \Omega_{\partial V}\big) = \gamma(V,U) .
\end{equation}
We say that such a coupling is optimal for the marginals on $U$. In particular, such couplings exist.

\begin{remark}\label{rem:optimal-couplings}
A word of caution regarding the optimality of grand couplings: there typically does not exist an optimal coupling of $\mu_1,\dots,\mu_k$ which induces an optimal coupling of every subcollection $\mu_{i_1},\dots,\mu_{i_m}$. As an example, consider the case where $\mu_i$, for $1 \le i \le k$, is the uniform distribution on $\{1,2,\dots,k\} \setminus \{i\}$. Observe that any coupling of $\mu_1,\dots,\mu_k$ is optimal since the right-hand side of~\eqref{eq:def-optimal-coupling} is zero for this choice of $\mu_1,\dots,\mu_k$. Moreover, one may check that if $(Z_1,\dots,Z_k)$ is sampled from any coupling of $\mu_1,\dots,\mu_k$, then $\Pr(Z_i \neq Z_j) \ge 2/k$ for some $i \neq j$. In particular, since $\|\mu_i-\mu_j\|=1/(k-1)$ for $i \neq j$, there does not exist a grand coupling for which $\Pr(Z_i \neq Z_j) = \|Z_i - Z_j\|$ for all $i \neq j$. In fact, the same is true for any $2 \le m \le k-1$. Namely, for any grand coupling, $\Pr(Z_{i_1} = \dots = Z_{i_m}) \le 1-m/k$ for some $1 \le i_1<\dots<i_m \le k$, which is strictly less than $\sum_{j=1}^k \min \{\Pr(Z_{i_1}=j),\dots,\Pr(Z_{i_m}=j)\}=1-(m-1)/(k-1)$. Thus, no grand coupling induces an optimal coupling of $\mu_{i_1},\dots,\mu_{i_m}$ for all $1 \le i_1<\dots<i_m \le k$.

This example shows that it is not obvious that one can find a grand coupling of $(P^\tau_V)_{\tau \in \Omega_{\partial V}}$ which reflects ``good'' properties the subcollections $(P^\tau_V)_{\tau \in \cT}$, $\cT \subset \Omega_{\partial V}$. This will be a primary difficulty in the proof of part~\eqref{thm:main-b} of Theorem~\ref{thm:main} in Section~\ref{sec:strong-mixing}, where we will use strong spatial mixing, a property stated in terms of pairs of boundary conditions, to construct a grand coupling which has a similar property for all pairs (and other $\cT$) simultaneously.

Lastly, we mention that it has recently been shown~\cite{angel2019pairwise} that there always exists a grand coupling which induces \emph{nearly} optimal pairwise couplings, but that this is not sufficient for our purposes, as we also need the induced coupling to be nearly optimal for the marginals on arbitrary $U \subset V$.
\end{remark}

\subsection{The coupled dynamics}
\label{sec:coupled-dynamics}

We shall require a simultaneous coupling of the dynamics started from all possible starting states. Let us define these coupled dynamics formally.

The $n$-th step of the dynamics depends on three parameters: a number $p_n \in (0,1)$ indicating the probability to be active, a base block $\Delta_n \subset \Z^d$ and a grand coupling $\pi_n$ of $(P^\tau_{\Delta_n})_{\tau \in \Omega_{\partial\Delta_n}}$.
Given these parameters, we let $((A_n,B_n))_{n \ge 1}$ be a sequence of independent random variables, where, for each $n \ge 1$, $A_n$ and $B_n$ are independent, $A_n$ is a Bernoulli random variable with parameter $p_n$ and $B_n = (B_n^\tau)_{\tau \in \Omega_{\partial \Delta_n}}$ is a sample from $\pi_n$.
We place an independent copy of this sequence at each vertex of $\Z^d$ so as to obtain an \iid\ process
\[ Y=(Y_v)_{v \in \Z^d}, \qquad\text{where }Y_v := ((A_{v,n},B_{v,n}))_{n\ge 1} .\]
This will ultimately be the process from which we will obtain a finitary coding to our target Markov random field $X$.

We regard $A_{v,n}$ as indicating whether $v$ is active at time $n$. We then let $A'_{v,n}$ be the indicator of the event that $v$ is the only active vertex in a ball of radius
\begin{equation}\label{eq:r_n}
 r_n := \diam(\Delta_n)+1
\end{equation}
around $v$ at time~$n$, so that $A'_{v,n}$ indicates whether $v$ is chosen at time $n$.
Define
\begin{equation}\label{eq:U}
U_{v,n} := \Big\{ u \in v-\Delta_n : A'_{u,n}=1 \Big\} ,
\end{equation}
and observe that $|U_{v,n}| \le 1$ almost surely.
If $U_{v,n}=\emptyset$ then $v$ will not be updated at time $n$, while, if $U_{v,n}=\{u\}$, then $v$ will be updated as part of the update to the block centered at $u$.

This and the description in Section~\ref{sec:dynamics} then give rise to the random function
\[ f_n \colon \Omega \to \Omega \]
defined by
\begin{equation}\label{eq:dynamics}
f_n(\xi)_v :=
\begin{cases}
	\xi_v &\text{if }U_{v,n} = \emptyset\\
	B^\tau_{u,n}(v-u) &\text{if }U_{v,n}=\{u\}\text{ and }\tau=(\xi_{w+u})_{w \in \partial \Delta_n}
\end{cases} .
\end{equation}
It is immediate that $f_n$ is almost surely well-defined.
The function $f_n$ describes a coupling of one step of the dynamics (at time $n$) started from all possible starting states (see~\cite{diaconis1999iterated} for more on this point of view). We remark that these dynamics may be seen as a (non-time-homogeneous) probabilistic cellular atomaton (PCA).

\begin{lemma}\label{lem:distribution-preserved}
	For any $n \ge 1$, $f_n$ preserves the distribution of $X$.
	More precisely, if $X$ and $Y$ are independent, then $f_n(X)$ has the same distribution as $X$.
\end{lemma}
\begin{proof}
The proof is standard (and straightforward) and thus omitted.
\end{proof}
%\begin{proof}
%	Denote $\sigma := f_n(\xi)$ and $A:=(A_{v,n})_{v \in \Z^d}$.
%	It suffices to show that, conditioned on $A$, $\sigma$ has the same distribution as $\xi$. Denote $A' := \{ v \in \Z^d : A'_{v,n}=1\}$ and note that $A'$ is measurable with respect to $A$.
%	Let $V' \subset \Z^d \setminus (A' + \Lambda_n^+)$ and $U \subset A'$ be finite and denote $\Lambda := V' \cup (U + \Lambda_n^+)$. Denote $V := V' \cup \partial(U+\Lambda_n)$ so that $\Lambda = V \cup (U+\Lambda_n)$. Let $\tau \in S^\Lambda$. Then
%	\begin{align*}
%	\Pr(\sigma_\Lambda = \tau \mid A)
%	&= \Pr(\xi_V=\tau_V,~\sigma_{\Lambda \setminus V}=\tau_{\Lambda \setminus V} \mid A) \\
%	&= \Pr(\xi_V=\tau_V) \cdot \Pr(\sigma_{\Lambda \setminus V}=\tau_{\Lambda \setminus V} \mid \xi_V=\tau_V,~A) \\
%	&= \Pr(\xi_V=\tau_V) \cdot \prod_{u \in U} \Pr\Big(B_{u,n}^{(\tau_{w+u})_{w \in \partial \Lambda_{\ell_n}}} = (\tau_{w+u})_{w \in \Lambda_{\ell_n}}\Big)\\
%	&= \Pr(\xi_V=\tau_V) \cdot \prod_{u \in U} P_{u+\Lambda_{\ell_n}}^{\tau_{u+\partial \Lambda_{\ell_n}}}(\tau_{u+\Lambda_{\ell_n}})\\
%	&= \Pr(\xi_\Lambda = \tau) . \qedhere
%	\end{align*}
%\end{proof}

\subsection{Constructing finitary codings using coupling-from-the-past}
\label{sec:cftp}

Let $\xi \in \Omega$ be the starting configuration.
The usual dynamics is loosely described by the evolution:
\[ \xi ~~\longrightarrow~~ f_1(\xi) ~~\longrightarrow~~ (f_2 \circ f_1)(\xi) ~~\longrightarrow~~ \cdots ~~\longrightarrow~~ (f_n \circ f_{n-1} \circ \cdots \circ f_2 \circ f_1)(\xi) ~~\longrightarrow~~ \cdots .\]
In other words, one starts with the configuration $\xi$ at time zero and then, by iteratively composing functions on the left, follows its trajectory forward in time until some positive time $n$.
The method of coupling-from-the-past suggests to reverse this process by starting with the configuration $\xi$ at some negative time and then following its trajectory (still forward in time) until time zero. Rather than using negative times, one may instead realize this idea by composing functions on the right. This leads us to consider the evolution:
\[  \cdots ~~\longleftarrow~~ (f_1 \circ f_2 \circ \cdots \circ f_{n-1} \circ f_n)(\xi) ~~\longleftarrow~~ \cdots ~~\longleftarrow~~ (f_1 \circ f_2)(\xi) ~~\longleftarrow~~ f_1(\xi) ~~\longleftarrow~~ \xi .\]
The difference between the two evolutions is subtle, but significant.
In the first process, when $(f_n \circ f_{n-1} \circ \cdots \circ f_2 \circ f_1)(\cdot)_v$ becomes a constant (but still random) function, even though the evolution ceases to depend on the starting configuration, it still continues to evolve in time. The main advantage of the second process is that once $(f_1 \circ f_2 \circ \cdots \circ f_{n-1} \circ f_n)(\cdot)_v$ becomes constant, the evolution not only becomes independent of the starting configuration, but also stabilizes.

The following proposition summarizes the usefulness of coupling-from-the-past for our purposes.
Denote
\[ \omega^n := f_1 \circ f_2 \circ \cdots \circ f_{n-1} \circ f_n \]
and
\begin{equation}\label{eq:coalescence-time}
T_v := \min \big\{ n \ge 0 : (\omega^n(\cdot))_v\text{ is a constant function}\big\}.
\end{equation}

\begin{prop}\label{prop:coupling-from-past}
	Suppose that $T_v$ is almost surely finite for all $v$. Then the random field
	\[ \omega^*=(\omega^{T_v}(\cdot)_v)_{v \in \Z^d} \]
	has the same distribution as $X$.
\end{prop}

\begin{proof}
	Let $\xi$ be sampled from the distribution of $X$, independently of $Y$. Fix a finite $V \subset \Z^d$.	
	Since $\omega^t(\cdot)_v=\omega^*_v$ for all $t \ge T_v$, we have
	$\omega^t(\xi)_V = \omega^*_V$ for all $t \ge \max_{v \in V} T_v$.
	Since, by Lemma~\ref{lem:distribution-preserved}, $\omega^t(\xi)_V$ has the same distribution as $X_V$ for all $t \ge 0$, and since $\omega^t(\xi)_V=\omega^*_V$ for sufficiently large~$t$, we conclude that $\omega^*_V$ has the same distribution as $X_V$.
As $V$ was arbitrary, the proposition follows.
\end{proof}

\begin{prop}\label{prop:coding-radius}
Suppose that $T_v$ is almost surely finite for all $v$. Then $\omega^*$ defines a finitary coding from $Y$ to $X$ whose coding radius $R$ satisfies
\[ R \le 2(r_1 + r_2 + \cdots + r_{T_\zero}) .\]
\end{prop}
\begin{proof}
The fact that $\omega^*$ defines a coding from $Y$ to $X$ is a consequence of Proposition~\ref{prop:coupling-from-past} and the observation that $f_n$ is translation-equivariant by~\eqref{eq:dynamics}.
We proceed to bound the coding radius.
By~\eqref{eq:U} and~\eqref{eq:dynamics}, for any $n \ge 1$ and $v \in \Z^d$, $f_n(\xi)_v$ depends only on $\xi_{v+\Lambda_{2r_n}}$ and $(Y_{v+u,n})_{u \in \Lambda_{2r_n}}$. Thus, for any $k \ge 1$, $f_n(\xi)_{\Lambda_k}$ depends only on $\xi_{\Lambda_{k+2r_n}}$ and $(Y_{u,n})_{u \in \Lambda_{k+2r_n}}$. Proceeding by induction, we see that
\begin{equation}\label{eq:dependency}
\omega^n(\xi)_\zero\text{ depends only on }\xi_{\Lambda_{2r_1+\cdots+2r_n}}\text{ and }(Y_{u,i})_{u \in \Lambda_{2r_1+\cdots+2r_i}, 1 \le i \le n} .
\end{equation}
The bound on $R$ follows.
\end{proof}

\subsection{Outline of proof of theorems}
\label{sec:outline}

Proposition~\ref{prop:coding-radius} provides a recipe for constructing a finitary coding from $Y$ to $X$.
Thus, the method of proof of both parts of Theorem~\ref{thm:main} (and of Theorem~\ref{thm:ffiid}) consists of showing that a suitable choice of $(p_n,\Delta_n,\pi_n)_{n \ge 1}$ yields a coding whose coding radius has the desired properties (either power-law tails for part~\eqref{thm:main-a} of the theorem or exponential tails for part~\eqref{thm:main-b}).

For part~\eqref{thm:main-a} of the theorem, we shall choose a rapidly growing sequence of blocks $\Delta_n$. The probabilities $p_n$ will be chosen to be of order $|\Delta_n|^{-1}$ so that a vertex has a constant probability of being updated at any step. Moreover, the grand coupling $\pi_n$ will be chosen so that there is also a constant probability to reach the stopping time $T_\zero$ at each step -- the existence of which is made possible by the assumption~\eqref{eq:positive-coupling-prob}. This will give exponential tails for $T_\zero$, but since $r_n$ will grow exponentially fast, it will yield heavy power-law tails for the coding radius.

For part~\eqref{thm:main-b} of the theorem, we shall choose $\Delta_n$ to be a large, but fixed (i.e., independent of~$n$), box $\Lambda_{n_0}$. The probabilities $p_n$ will also be fixed, say to $1/2$. The choice of $n_0$, and of course the assumption of exponential strong spatial mixing, will guarantee the existence of a grand coupling with ``contracting'' properties, which will in turn ensure that the stopping time $T_\zero$ is reached exponentially fast, thus leading to exponential tails for the coding radius.

\section{Weak spatial mixing implies \ffiid\ with power-law tails}
\label{sec:weak-mixing}

In this section, we prove Theorem~\ref{thm:ffiid} and part~\eqref{thm:main-a} of Theorem~\ref{thm:main}. Let us first explain how Theorem~\ref{thm:ffiid} implies part~\eqref{thm:main-a} of Theorem~\ref{thm:main}.
For this, we require an additional notion of spatial mixing.
A Markov random field and its specification are said to satisfy \emph{ratio weak spatial mixing} with rate function $\rho$ if, for any finite sets $U \subset V \subset \Z^d$,
\begin{equation}\label{eq:ratio-WSM-def}
\max_{\omega \in \Omega_U} \left(1 - \frac{P^{\tau'}_{V,U}(\omega)}{P^\tau_{V,U}(\omega)} \right) \le |U| \cdot \rho(\dist(U,\partial V)) \qquad\text{for any }\tau,\tau' \in \Omega_{\partial V},
\end{equation}
where the maximum is only over those $\omega$ having $P^\tau_{V,U}(\omega)>0$.
It may strike the reader as strange that there are parenthesis in the left-hand side of~\eqref{eq:ratio-WSM-def} and not absolute value. However, the two are essentially equivalent, since when the right-hand side of~\eqref{eq:ratio-WSM-def} is, say, less than $1/2$, the parenthesis can be replaced with absolute value at the price of introducing a multiplicative factor of 2 in the right-hand side.
The difference between the two is then only apparent when the right-hand side is not small. In particular, the definition given in~\eqref{eq:ratio-WSM-def} does not convey any information when the right-hand side is greater than 1. For this reason, we find this definition more convenient.
Ratio weak spatial mixing is a priori stronger than weak spatial mixing, as~\eqref{eq:ratio-WSM-def} clearly implies~\eqref{eq:WSM-def}. Surprisingly, exponential weak spatial mixing is equivalent to exponential ratio weak spatial mixing (see~\cite[Theorem~3.3]{alexander1998weak} or Section~\ref{sec:ratio-SSM} below for a short self-contained proof of an analogous statement for strong spatial mixing).
Observe that exponential ratio weak spatial mixing easily implies that, for any fixed $0<\delta<1$, for sufficiently large $n$, we have
\begin{equation}\label{eq:delta-ratio-WSM-def}
\min_{\tau,\tau' \in \Omega_{\partial \Lambda_n}} \min_{\omega \in \Omega_{\Lambda_{\delta n}}} ~ \frac{P^{\tau'}_{\Lambda_n,\Lambda_{\delta n}}(\omega)}{P^\tau_{\Lambda_n,\Lambda_{\delta n}}(\omega)} \ge \delta ,\end{equation}
where the second minimum is over $\omega$ having $P^\tau_{\Lambda_n,\Lambda_{\delta n}}(\omega)>0$ (in fact, the left-hand side in~\eqref{eq:delta-ratio-WSM-def} tends to 1 exponentially fast as $n \to \infty$).
It is immediate that~\eqref{eq:delta-ratio-WSM-def} implies~\eqref{eq:positive-coupling-prob}.

We now turn to the proof of Theorem~\ref{thm:ffiid}.
Let $X$ be a translation-invariant Markov random field satisfying~\eqref{eq:positive-coupling-prob2}. It follows from~\eqref{eq:positive-coupling-prob2} that there exist $0<\delta,\epsilon<1/3$ and a sequence $(\ell_n)_{n \ge 1}$ of positive integers such that, for all $n \ge 1$,
\begin{equation}\label{eq:ell_n}
\delta \ell_{n+1} > 4d(\ell_1+\cdots+\ell_n)
\end{equation}
and
\begin{equation}\label{eq:RSW2}
\gamma(\Lambda_{\ell_n},\Lambda_{3\delta \ell_n}) > \epsilon.
\end{equation}
In addition, if~\eqref{eq:positive-coupling-prob} holds, then $(\ell_n)$ can be chosen to also satisfy
\begin{equation}\label{eq:ell_n_exp_growth}
\ell_n \le e^{cn} \qquad\text{for some }c>0\text{ and all }n \ge 1.
\end{equation}
Set $\Delta_n := \Lambda_{\ell_n}$ and $p_n := \ell_n^{-d}$. Recall~\eqref{eq:r_n} and note that $r_n = 2d\ell_n$.
Let $\pi_n$ be any grand coupling of $(P^\tau_{\Delta_n})_{\tau \in \Omega_{\partial \Delta_n}}$ which is optimal for the marginals on $\Lambda_{3\delta \ell_n}$, so that, by~\eqref{eq:RSW2} and~\eqref{eq:optimal-coupling},
\begin{equation}\label{eq:coupling}
\pi_n(\cE_n) > \epsilon , \qquad\text{where }\cE_n := \Big\{ \omega^\tau_{\Lambda_{3\delta \ell_n}}=\omega^{\tau'}_{\Lambda_{3\delta \ell_n}}\text{ for all }\tau,\tau' \in \Omega_{\partial \Delta_n} \Big\} .
\end{equation}

\begin{figure}
	\centering
	\begin{subfigure}[t]{.31\textwidth}
		\centering
		\includegraphics[scale=0.52]{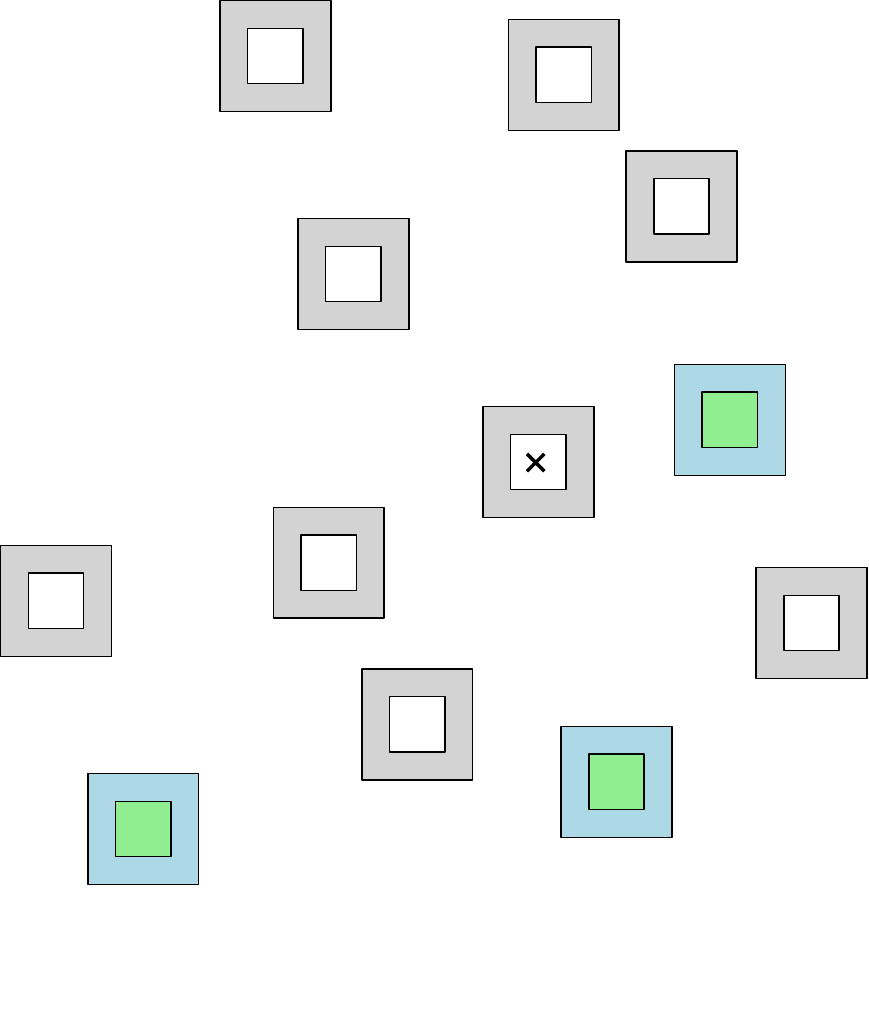}
		\caption{$n=1$.}
		\label{fig:dynamics1}
	\end{subfigure}%
	\begin{subfigure}{15pt}
		\quad
	\end{subfigure}%
	\begin{subfigure}[t]{.31\textwidth}
		\centering
		\includegraphics[scale=0.52]{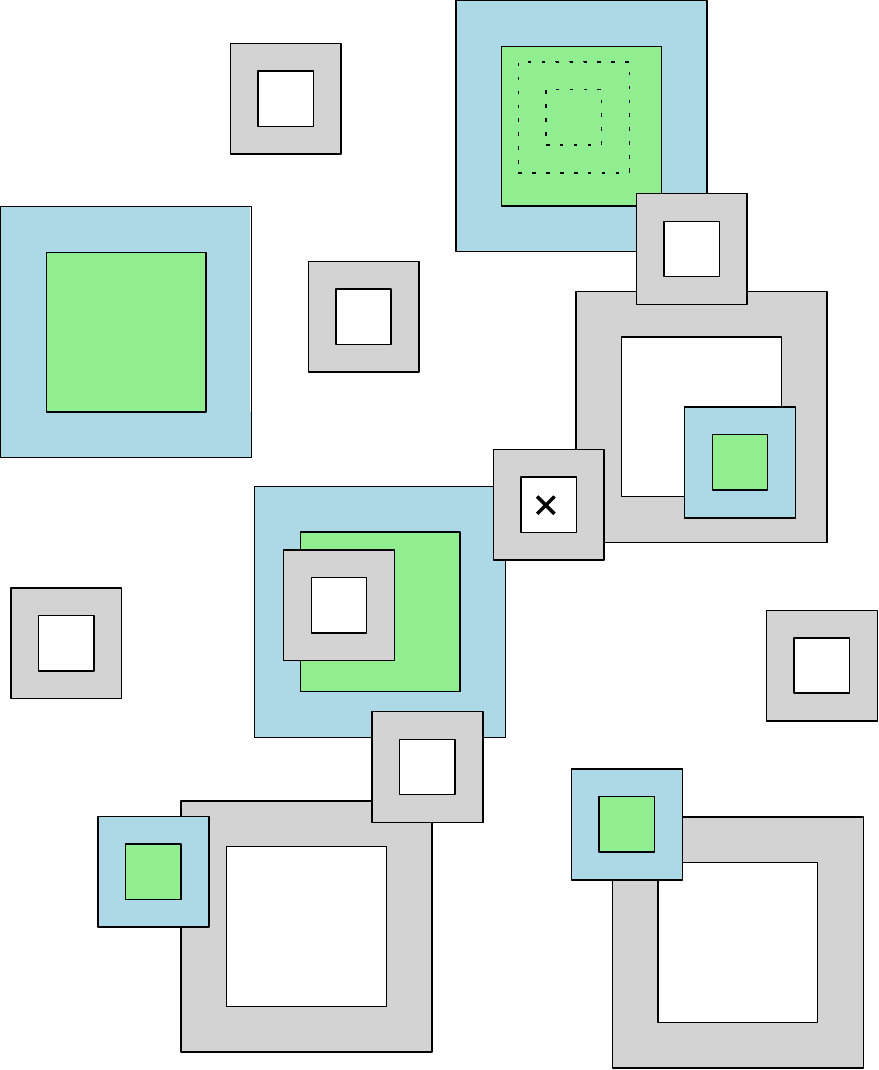}
		\caption{$n=2$.}
		\label{fig:dynamics2}
	\end{subfigure}%
	\begin{subfigure}{15pt}
		\quad
	\end{subfigure}%
	\begin{subfigure}[t]{.31\textwidth}
		\centering
		\includegraphics[scale=0.52]{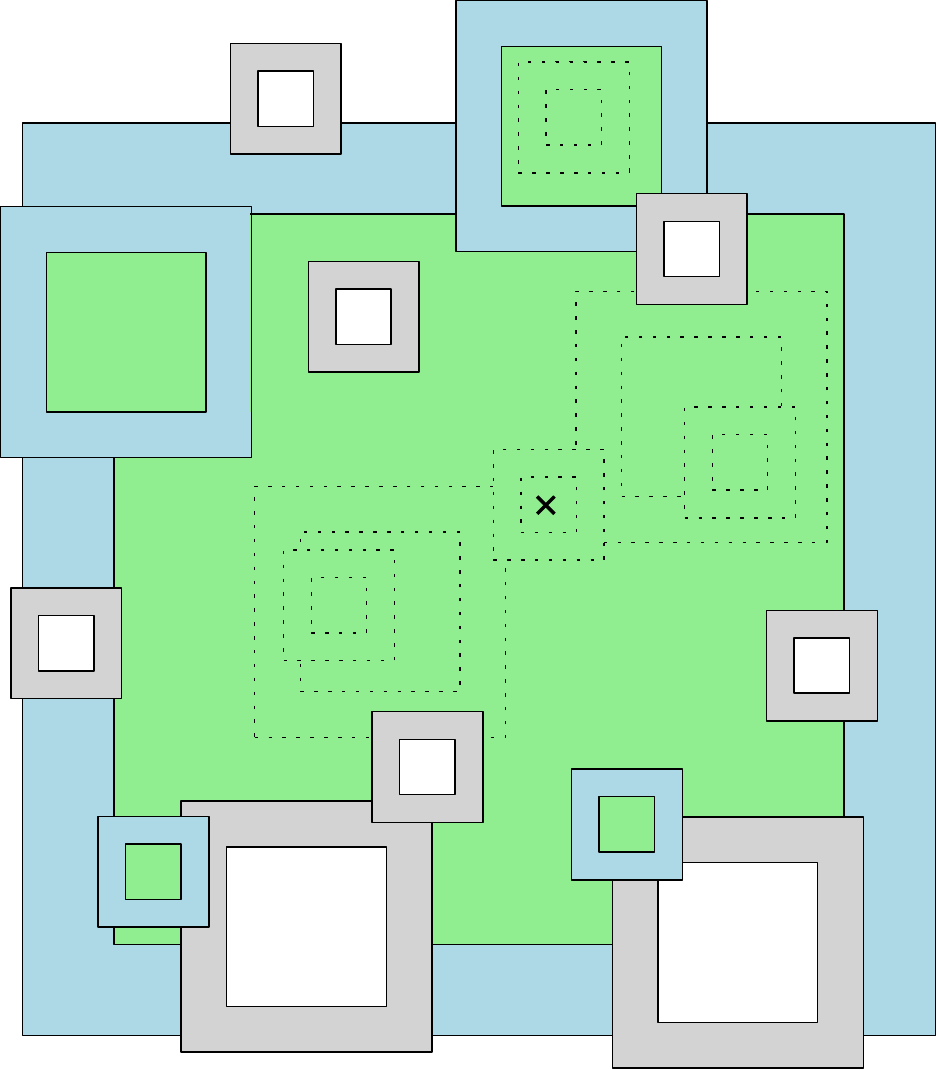}
		\caption{$n=3$.}
		\label{fig:dynamics3}
	\end{subfigure}
	\caption{At each step of the dynamics, the blocks of the chosen vertices (i.e., those $v$ having $A'_{v,n}=1$) are updated. The updated blocks $v+\Lambda_{\ell_n}$ and their inner regions $v+\Lambda_{3\delta \ell_n}$ are depicted (not drawn to scale). The blocks which successfully shield their inner regions from the outside boundary condition (i.e., those for which $B_{v,n} \in \cE_n$) are depicted in light blue and the ones which fail to do so in gray.
	The vertex $\zero$ is indicated by $\times$. The green shade indicates regions of vertices $v$ for which $T'_v \le n$, that is, vertices whose value at time 0 is guaranteed not to depend on the starting configuration at time $n$. In particular, $T'_\zero$ is the first time at which the origin becomes green, which is $T'_\zero=3$ in the depicted case.}
	\label{fig:dynamics}
\end{figure}

In light of Proposition~\ref{prop:coding-radius}, Theorem~\ref{thm:ffiid} boils down to establishing bounds on the tails of $T_v$. These are provided by the following two lemmas.
Define
\[ T'_v := \min \Big\{ n \ge 0 : U_{v,n}=\{u\}\text{ and }B_{u,n}\in\cE_n\text{ for some $u \in v+\Lambda_{\delta \ell_n}$} \Big\} .\]
In words, $T'_v$ is the first time at which $v$ is updated by the dynamics in such a way that (i) it falls deep within the inner region of the updated block and (ii) the updated block successfully couples its inner region (thus shielding its inner region from the boundary conditions). See Figure~\ref{fig:dynamics}.

\begin{lemma}
	$T'_v$ has exponential tails. In particular, $T'_v$ is almost surely finite.
\end{lemma}
\begin{proof}
	Fix $n \ge 1$. Let $\cF_n$ be the event that $U_{\zero,n}=\{u\}$ and $B_{u,n}\in\cE_n$ for some $u\in \Lambda_{\delta \ell_n}$. Then, using that $\{(A_{u,n},B_{u,n})\}_{u \in \Z^d}$ are \iid, that $A_{u,n}$ and $B_{u,n}$ are independent, and~\eqref{eq:U},
\begin{align*}
	\Pr(\cF_n) &\ge \Pr\big(B_{\zero,n}\in\cE_n\big) \cdot \Pr\Big(\big|\big\{u \in \Lambda_{\delta \ell_n} : A_{u,n}=1\big\}\big| = 1 \Big) \cdot \Pr\Big(\big\{u \in \Lambda_{\delta\ell_n+r_n} : A_{u,n}=1\big\} = \emptyset \Big) \\
	&= \pi_n(\cE_n) \cdot \Pr\Big(\text{Bin}\big(|\Lambda_{\delta \ell_n}|, \ell_n^{-d}\big) = 1\Big) \cdot \Pr\Big(\text{Bin}\big(|\Lambda_{\delta\ell_n+r_n}|, \ell_n^{-d}\big) = 0\Big) \ge c > 0 ,
\end{align*}
where we used~\eqref{eq:coupling} and the fact that $|\Lambda_{\delta \ell_n}|=\Theta(\ell_n^d)$ and $|\Lambda_{\delta \ell_n+r_n}|=\Theta(\ell_n^d)$.
Thus, by translation invariance, $\Pr(T'_v>n) = \Pr(T'_\zero>n) = \Pr(\cF_1^c) \cdots \Pr(\cF_n^c) \le (1-c)^n$.
\end{proof}

\begin{lemma}
$T_v \le T'_v$ almost surely.
\end{lemma}
\begin{proof}
By translation invariance, it suffices to prove that $T_\zero \le T'_\zero$ almost surely.
By the definition of $T_\zero$, it suffices to show that $\omega^{T'_\zero}(\xi)_\zero$ does not depend on the starting state $\xi$.

By~\eqref{eq:dependency}, we have that $\omega^{n-1}(\xi)_\zero$ depends on $\xi$ only through $\xi_{\Lambda_{2r_1+\dots+2r_{n-1}}}$. By~\eqref{eq:ell_n} and since $r_n=2d\ell_n$, we have that $\omega^{n-1}(\xi)_\zero$ depends on $\xi$ only through $\xi_{\Lambda_{\delta \ell_n}}$.
Since $\omega^n(\xi)_\zero = \omega^{n-1}(f_n(\xi))_\zero$, we conclude that $\omega^n(\xi)_\zero$ depends on $\xi$ only through $f_n(\xi)_{\Lambda_{\delta \ell_n}}$.

Suppose now that $U_{\zero,n}=\{u\}$ and $B_{u,n}\in\cE_n$ for some $n \ge 1$ and $u\in \Lambda_{\delta \ell_n}$.
Observe that $U_{v,n} = \{u\}$ for all $v \in u+\Lambda_{\ell_n}$. Thus, by~\eqref{eq:dynamics}, $(f_n(\xi)_{w+u})_{w \in \Lambda_{\ell_n}} = B^\tau_{u,n}$ for some $\tau \in \Omega_{\partial \Lambda_{\ell_n}}$ which depends on $\xi$. But since $B_{u,n}\in\cE_n$, \eqref{eq:coupling} implies that $(B^\tau_{u,n})_{\Lambda_{3\delta \ell_n}}$ does not depend on $\tau$. Hence, $(f_n(\xi)_{w+u})_{w \in \Lambda_{3\delta \ell_n}}$, or equivalently, $f_n(\xi)_{u+\Lambda_{3\delta \ell_n}}$, does not depend on $\xi$. Since $u \in \Lambda_{\delta \ell_n}$, we have $\Lambda_{\delta \ell_n} \subset u+\Lambda_{3\delta \ell_n}$, and thus, $f_n(\xi)_{\Lambda_{\delta \ell_n}}$ does not depend on $\xi$.
We conclude that $\omega^{T'_\zero}(\xi)_\zero$ does not depend on $\xi$.
\end{proof}

\begin{proof}[Proof of Theorem~\ref{thm:ffiid}]
In light of Proposition~\ref{prop:coding-radius}, the first part of Theorem~\ref{thm:ffiid} follows immediately from the fact that $T_v$ is almost surely finite, as is shown by the above two lemmas. Similarly, for the second part of the theorem, it suffices to show that~\eqref{eq:positive-coupling-prob} implies that $r_1+r_2+\dots+r_{T_v}$ has power-law tails. Indeed, since the above two lemmas imply that $\Pr(T_v>t) \le Ce^{-ct}$ for some $C,c>0$ and all $t>0$, and since~\eqref{eq:ell_n} and~\eqref{eq:ell_n_exp_growth} imply that $r_1+r_2+\dots+r_n \le e^{c'n}$ for some $c'>0$ and all $n \ge 1$, we see that
\[ \Pr(r_1+r_2+\dots+r_{T_v} > t) \le \Pr\big(e^{c'T_v} > t\big) = \Pr\big(T_v > \tfrac{1}{c'} \log t\big) \le Ce^{-\tfrac{c}{c'} \log t} = C t^{-c/c'} . \qedhere \]
\end{proof}

\section{Strong spatial mixing implies \ffiid\ with exponential tails}
\label{sec:strong-mixing}

In this section, we prove part~\eqref{thm:main-b} of Theorem~\ref{thm:main}.
Let us begin with a simple observation.
For a finite set $V \subset \Z^d$ and a grand coupling $\pi$ of $(P^\tau_V)_{\tau \in \Omega_{\partial V}}$, define
\[ \kappa_\pi := \frac{1}{|V|} \sum_{v \in V} \sum_{s \in S} \pi\big(\omega^\tau_v = s\text{ for all }\tau \in \Omega_{\partial V} \big) .\]
One may regard $\kappa_\pi$ as the average ratio of vertices in $V$ that are successfully coupled (i.e., whose value does not depend on the boundary condition $\tau$) by $\pi$.
As such, recalling~\eqref{eq:optimal-coupling}, it is easy to see that every optimal grand coupling $\pi$ satisfies
\[ \kappa_\pi \ge \gamma(V,U) \cdot \frac{|U|}{|V|} \qquad\text{for any }U \subset V.\]
We aim to prove a statement of the form ``if there exists a grand coupling $\pi$ for which the average ratio of vertices in $V$ that are successfully coupled by $\pi$ is sufficiently close to 1, then $X$ is \ffiid\ with exponential tails''. In particular, we will show that if there exists a grand coupling $\pi$ for which
\begin{equation}\label{eq:kappa-condition}
\kappa_\pi > 1 - \frac{1}{|\partial V|} ,
\end{equation}
then $X$ is \ffiid\ with exponential tails (showing that the high noise condition~\eqref{eq:high-noise} can be somewhat relaxed).
Unfortunately, this result is not sufficient in order to deduce part~\eqref{thm:main-b} of Theorem~\ref{thm:main} as, under the assumption of exponential strong spatial mixing, we cannot guarantee that~\eqref{eq:kappa-condition} holds for some grand coupling. We thus require a more refined version of this, which we now formulate.

For a boundary condition $\tau \in \Omega_{\partial V}$ and a subset $A \subset \partial V$, define
\[ \lambda_\pi(\tau,A) := \frac{1}{|V|} \sum_{v \in V} \pi\Big(\omega^\tau_v \neq \omega^{\tau'}_v\text{ for some } \tau' \in \Omega_{\partial V}\text{ having }\tau_{\partial V \setminus A}=\tau'_{\partial V\setminus A} \Big) .\]
Note that $\lambda_\pi(\tau,A) \le \lambda_\pi(\tau,A')$ for any $\tau \in \Omega_{\partial V}$ and $A \subset A' \subset \partial V$, and that
\[ \lambda_\pi(\tau,\partial V) = 1 - \kappa_\pi \qquad\text{for any }\tau \in \Omega_{\partial V} .\]
One may regard the role of $A$ in $\lambda_\pi(\tau,A)$ as indicating which vertices on the boundary $\partial V$ are still uncoupled (i.e., whose values depend on the initial configuration in the coupled dynamics and are thus unknown), so that $\lambda_\pi(\tau,A)$ measures the average ratio of vertices in $V$ that are not successfully coupled by~$\pi$ (and hence, whose values are unknown after applying one step of the coupled dynamics), given that we know that the values of the vertices in $\partial V \setminus A$ are given by $\tau$.
Thus, $\lambda_\pi(\tau,A)$ allows us to relax~\eqref{eq:kappa-condition}, or equivalently, $\lambda_\pi(\tau,A) < 1/|\partial V|$, to a condition which takes into account the known/available information. When little or no information is available (i.e., when $A$ is large or even equal to $\partial V$), we shall require a much more modest bound on $\lambda_\pi(\tau,A)$. As more information becomes available (i.e., as $A$ becomes smaller), we shall require a better bound, which will match the original one precisely when $|A|=1$.

\begin{thm}\label{thm:contraction-implies-ffiid}
	Let $X$ be a translation-invariant Markov random field with specification $P$.
	If for some finite set $V\subset \Z^d$ and some grand coupling $\pi$ of $(P^\tau_V)_{\tau \in \Omega_{\partial V}}$,
	\begin{equation}\label{eq:good-coupling-prob2}
	\lambda_\pi(\tau,A) < \frac{|A|}{|\partial V|} \qquad\text{for all }\tau \in \Omega_{\partial V}\text{ and non-empty }A \subset \partial V,
	\end{equation}
	then $X$ is \ffiid\ with exponential tails.
\end{thm}

One may regard a coupling satisfying~\eqref{eq:good-coupling-prob2} as a contraction, since it reduces the density of vertices whose values are unknown. Indeed, \eqref{eq:good-coupling-prob2} says that if we know the values of the vertices in $\partial V \setminus A$, then, on average, after applying one step of the coupled dynamics, the ratio of vertices in $V$ whose values are unknown is strictly less than it was on the boundary.

The following theorem implies that, under the assumption of exponential strong spatial mixing, one may always find a grand coupling satisfying~\eqref{eq:good-coupling-prob2}, at least when $V$ is a sufficiently large box.

\begin{thm}\label{thm:grand-coupling-exists}
	Let $P$ be the specification of a Markov random field satisfying exponential strong spatial mixing.
	Then there exists a grand coupling $\pi_n$ of $(P^\tau_{\Lambda_n})_{\tau \in \Omega_{\partial \Lambda_n}}$ such that
	\[ \lambda_{\pi_n}(\tau,A) \le \frac{C\log^d n}{n^d} \cdot |A| \qquad\text{for all }\tau \in \Omega_{\partial \Lambda_n}\text{ and }A \subset \partial \Lambda_n ,\]
	where $C>0$ depends on $P$, but does not depend on $n$.
\end{thm}

Part~\eqref{thm:main-b} of Theorem~\ref{thm:main} is an immediate consequence of the two theorems above.
The proof of Theorem~\ref{thm:contraction-implies-ffiid} is given in Section~\ref{sec:contraction-implies-ffiid}.
In Section~\ref{sec:ratio-SSM}, we prove that exponential strong spatial mixing implies exponential ratio strong spatial mixing. We then use this in Section~\ref{sec:grand-coupling} to prove Theorem~\ref{thm:grand-coupling-exists}.

\subsection{Contraction implies \ffiid\ with exponential tails}
\label{sec:contraction-implies-ffiid}

It is instructive to consider a more general setting than that of Theorem~\ref{thm:contraction-implies-ffiid}.
Let $\psi \colon 2^S \setminus \{\emptyset\} \to [0,\infty)$ be any function satisfying
\begin{equation}\label{eq:psi-monotone}
\psi(A) \le \psi(B) \qquad\text{ for any non-empty }A \subset B \subset S
\end{equation}
and
\begin{equation}\label{eq:psi-zero}
\psi(A)=0 \iff |A|=1 \qquad\text{ for any non-empty }A \subset S .
\end{equation}
We think of $\psi$ as a measure of how much uncertainty there is about the value at a vertex. The simplest such $\psi$ (and the one which is relevant for Theorem~\ref{thm:contraction-implies-ffiid} and Theorem~\ref{thm:grand-coupling-exists}) is $\psi_1$ defined by
\[ \psi_1(A) := \1_{\{|A|>1\}} .\]

For finite $V \subset \Z^d$ and $\eta \in (2^S)^{\partial V}$, denote
\[ \Omega_{\partial V,\eta} := \big\{ \tau \in \Omega_{\partial V} : \tau_v \in \eta_v\text{ for all }v \in \partial V \big\} .\]
For a grand coupling $\pi$ of $(P^\tau_V)_{\tau \in \Omega_{\partial V}}$ and $\eta \in (2^S)^{\partial V}$ such that $\Omega_{\partial V,\eta}\neq \emptyset$, define
\begin{equation}\label{eq:lambda-def}
 \lambda_{\pi,\psi}(\eta) := \frac{1}{|V|} \sum_{v \in V} \pi\Big[\psi\big(\big\{\omega^\tau_v : \tau \in \Omega_{\partial V,\eta} \big\}\big)\Big] .
\end{equation}
Note that $\lambda_{\pi,\psi_1}(\eta) = \lambda_\pi(\tau,A)$, where $A := \{ v \in \partial V : |\eta_v|>1 \}$ and $\tau \in \Omega_{\partial V}$ is any element satisfying $\eta_v=\{\tau_v\}$ for all $v \in \partial V \setminus A$.
Thus, the following is a generalization of Theorem~\ref{thm:contraction-implies-ffiid}.

\begin{thm}\label{thm:finitary-coding-kappa}
	Let $X$ be a translation-invariant Markov random field with specification $P$.
	If for some finite set $V\subset \Z^d$, some grand coupling $\pi$ of $(P^\tau_V)_{\tau \in \Omega_{\partial V}}$ and some $\epsilon>0$,
			\begin{equation}\label{eq:good-coupling-prob3}
			\lambda_{\pi,\psi}(\eta) \le \frac{1-\epsilon}{|\partial V|} \sum_{v \in \partial V} \psi(\eta_v) \qquad\text{for all }\eta \in (2^S)^{\partial V}\text{ such that }\Omega_{\partial V,\eta} \neq \emptyset,
			\end{equation}
			then $X$ is \ffiid\ with exponential tails.
\end{thm}

\begin{proof}
	Let $V$, $\pi$ and $\epsilon$ be as in the statement of the theorem.
	Set $\Delta_n := V$, $p_n = 1/2$ and $\pi_n := \pi$ for all $n \ge 1$.
In light of Proposition~\ref{prop:coding-radius}, the theorem will follow once we show that $T_v$ has exponential tails.
Observe that $\omega^n=f_1 \circ\dots\circ f_n$ has the same distribution as $f_n \circ\dots\circ f_1$. Thus,
\[ \Pr(T_v > n) = \Pr(\omega^n(\cdot)_v\text{ is not constant}) = \Pr((f_n \circ\dots\circ f_1)(\cdot)_v\text{ is not constant}) .\]
Denote
\[ Z_n(v) := \big\{ (f_n \circ\dots\circ f_1)(\xi)_v : \xi \in \Omega \big\}\qquad\text{and}\qquad \Psi_n := \psi(Z_n(\zero)) .\]
By~\eqref{eq:psi-zero}, we have
\[ \Pr(T_\zero > n) = \Pr(|Z_n(\zero)| \neq 1) = \Pr(\Psi_n>0) .\]
Since $\psi$ takes finitely many values, it remains to show that $\E[\Psi_n]$ decays exponentially in $n$.

Let $\alpha := \Pr(A'_{v,n}=1)$ denote the probability that a vertex is chosen at any given time and let $\beta := \Pr(U_{v,n} \neq \emptyset) = \alpha |V|$ denote the probability that it is updated at any given time.
Let $\cF_n$ denote the $\sigma$-algebra generated by $\{ A_{v,i},B_{v,i} \}_{v \in \Z^d, 1 \le i \le n}$.
Then
\begin{align*}
\E[\Psi_{n+1} \mid \cF_n]
 &= \E\Big[ \E\big[\Psi_{n+1} \mid U_{\zero,n+1},~\cF_n\big] \mid \cF_n\Big] \\
 &= (1-\beta) \E\big[\Psi_{n+1} \mid U_{\zero,n+1}=\emptyset,~\cF_n\big] + \alpha \sum_{u \in V} \E\big[\Psi_{n+1} \mid U_{\zero,n+1}=\{-u\},~\cF_n\big] \\
 &= (1-\beta) \Psi_n + \alpha \sum_{u \in V} \pi\Big[\psi\big(\big\{\omega^\tau_u : \tau=((f_n \circ\dots\circ f_1)(\xi)_{w-u})_{w \in \partial V}\text{ for some }\xi \in \Omega \big\}\big) \Big] .
\end{align*}
Thus, by the definition of $Z_n$ and by~\eqref{eq:psi-monotone},
\[ \E[\Psi_{n+1} \mid \cF_n]
 \le (1-\beta) \Psi_n + \alpha \sum_{u \in V} \pi\Big[\psi\big(\big\{\omega^\tau_u : \tau \in \Omega_{\partial V,(Z_n(w-u))_{w \in \partial V}} \big\}\big) \Big] .\]
Using that the distribution of $Z_n$ is translation-invariant, \eqref{eq:lambda-def} and~\eqref{eq:good-coupling-prob3}, we obtain
\begin{align*}
\E[\Psi_{n+1}]
= \E \big[ \E[\Psi_{n+1} \mid \cF_n] \big]
&\le (1-\beta) \E[\Psi_n] + \beta \E \Big[\lambda_{\pi,\psi}\big(Z_n|_{\partial V}\big)\Big] \\
&\le (1-\beta) \E[\Psi_n] + \frac{\beta (1-\epsilon)}{|\partial V|} \sum_{v \in \partial V} \E\big[\psi\big(Z_n(v)\big)\big] = (1 - \epsilon\beta ) \E[\Psi_n] .
\end{align*}
Thus, $\E[\Psi_n]$ decays exponentially in $n$.
\end{proof}

\subsection{Ratio strong spatial mixing}
\label{sec:ratio-SSM}

Here we prove that strong spatial mixing implies ratio strong spatial mixing.
The proof is essentially contained in~\cite[Theorem~3.3]{alexander1998weak}, where it is shown (in a slightly more general setting than Markov random fields) that exponential weak spatial mixing implies exponential ratio weak spatial mixing (see also~\cite[Theorem~3.10]{alexander2004mixing} and~\cite{martinelli1999lectures}). We present a short proof with some simplifications.

We say that a Markov random field and its specification satisfy \emph{ratio strong spatial mixing} with rate function $\rho$ if, for any finite sets $U \subset V \subset \Z^d$,
\begin{equation}\label{eq:ratio-SSM-def}
\max_{\omega \in \Omega_U} \left(1- \frac{P^{\tau'}_{V,U}(\omega)}{P^\tau_{V,U}(\omega)} \right) \le |U| \cdot \rho(\dist(U,\Sigma_{\tau,\tau'})) \qquad\text{for any }\tau,\tau' \in \Omega_{\partial V},
\end{equation}
where the maximum is only over those $\omega$ having $P^\tau_{V,U}(\omega)>0$, and $\Sigma_{\tau,\tau'}$ was defined just after~\eqref{eq:SSM-def}.
Let $s_r$ be the number of vertices $v \in \Z^d$ such that $|v|=r$, and note that $s_r = \Theta(r^{d-1})$ as $r \to \infty$.
Given a rate function~$\rho$, define the rate function $\rho_*$ by
\begin{equation}\label{eq:def-rho-star}
\rho_*(r) := \begin{cases} \infty &\text{if $r=1$}\\3s_{\lfloor r/2 \rfloor} \cdot \sqrt{\rho\big(\big\lfloor \frac r2 \big\rfloor\big)} &\text{if }r \ge 2 \end{cases} . .
\end{equation}
Formally, $\rho_*(1)$ is required to be finite, but this is clearly not important (it could instead be taken to equal, say, $\max\{1,\rho^*(2)\}$), and we write $\infty$ just to emphasize that~\eqref{eq:ratio-SSM-def} has no content when $\dist(U,\Sigma_{\tau,\tau'})=1$.
This is not a mere artifact of the proof, and cannot be avoided in general.
Indeed, there are models with hard constraints (e.g., proper $q$-colorings with large $q$), which satisfy strong spatial mixing with a rate function $\rho$ with arbitrarily small $\rho(1)$, and for which the left-hand side of~\eqref{eq:ratio-SSM-def} equals 1 when $\dist(U,\Sigma_{\tau,\tau'})=1$, due to the hard constraints.

\begin{prop}\label{prop:SSM-implies-ratio-SSM}
	If a Markov random field satisfies strong spatial mixing with rate function $\rho$, then it satisfies ratio strong spatial mixing with rate function $\rho_*$.
\end{prop}
\begin{proof}
	Let $U \subset V \subset \Z^d$ be finite and let $\tau,\tau' \in \Omega_{\partial V}$. Denote $r := \dist(U,\Sigma_{\tau,\tau'})$ and $\epsilon := |U| \cdot \rho_*(r)$. Let $\omega$ and $\sigma$ denote samples from $P^\tau_V$ and $P^{\tau'}_V$, respectively.
We must show that
\[ \Pr(\sigma_U=\eta) \ge (1-\epsilon)\cdot \Pr(\omega_U=\eta) \qquad\text{ for all }\eta \in \Omega_U .\]
To this end, it suffices to show that one may couple $\omega$ and $\sigma$ so that
\begin{equation}\label{eq:H-prob}
\Pr\big(\omega_U=\sigma_U \mid \omega_U\big) \ge 1-\epsilon \qquad\text{almost surely} .
\end{equation}

Let us explain how to construct such a coupling. Since~\eqref{eq:H-prob} holds trivially when $\epsilon \ge 1$, we may assume that $\epsilon<1$. In particular, we have that $r_* :=\lfloor r/2 \rfloor \ge 1$. Denote
\[ B := \{ v \in V : \dist(v,U) = r_* \}\quad\text{and}\quad W := \{ v \in V : \dist(v,U) < r_* \} .\]
Observe that
\begin{equation}\label{eq:B-W-props}
 U \subset W, \qquad \partial W \subset B \cup (\partial V \setminus \Sigma_{\tau,\tau'}), \qquad \dist(B, U \cup \Sigma_{\tau,\tau'}) \ge r_* .
\end{equation}
We first sample $(\omega_B,\sigma_B)$ from an optimal coupling of $P^\tau_{V,B}$ and $P^{\tau'}_{V,B}$.
Next, we sample $(\omega_W,\sigma_W)$ from an optimal coupling of $P^\tau_V(\omega_W \in \cdot \mid \omega_B)$ and $P^{\tau'}_V(\sigma_W \in \cdot \mid \omega'_B)$.
Finally, we independently sample $\omega_{V \setminus (B \cup W)}$ from $P^\tau_V(\omega_{V \setminus (B \cup W)} \in \cdot \mid \omega_B)$ and $\sigma_{V \setminus (B \cup W)}$ from $P^{\tau'}_V(\sigma_{V \setminus (B \cup W)} \in \cdot \mid \sigma_B)$.
It is straightforward to verify that this indeed yields a sample $(\omega,\sigma)$ whose marginals are $P^\tau_V$ and $P^{\tau'}_V$.

\begin{figure}
	\centering
	\begin{subfigure}[t]{.4\textwidth}
		\centering
		\includegraphics[scale=0.65]{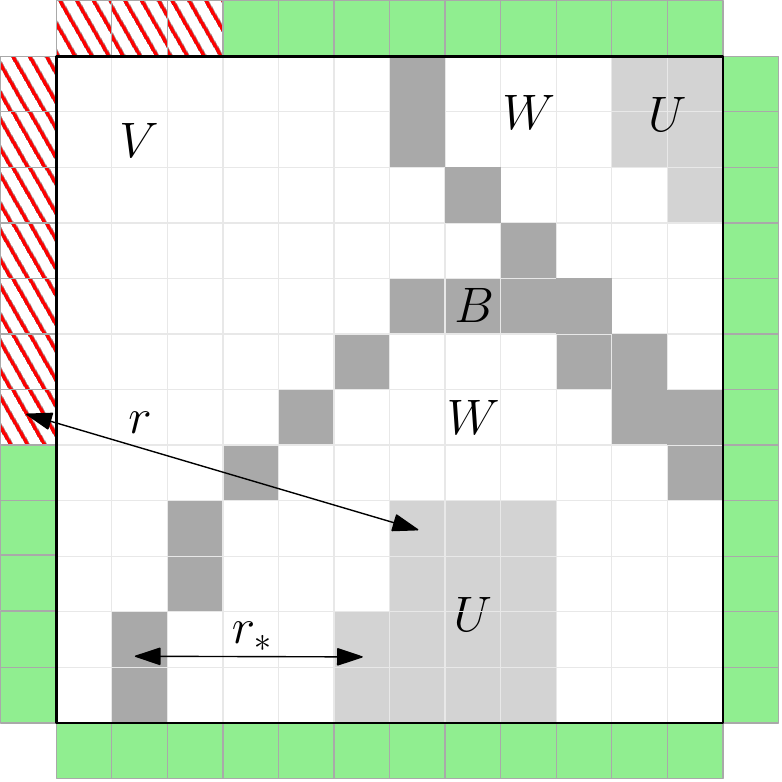}
		\caption{}
		\label{fig:ratio-mixing-a}
	\end{subfigure}%
	\begin{subfigure}{15pt}
		\quad
	\end{subfigure}%
	\begin{subfigure}[t]{.5\textwidth}
		\centering
		\includegraphics[scale=0.65]{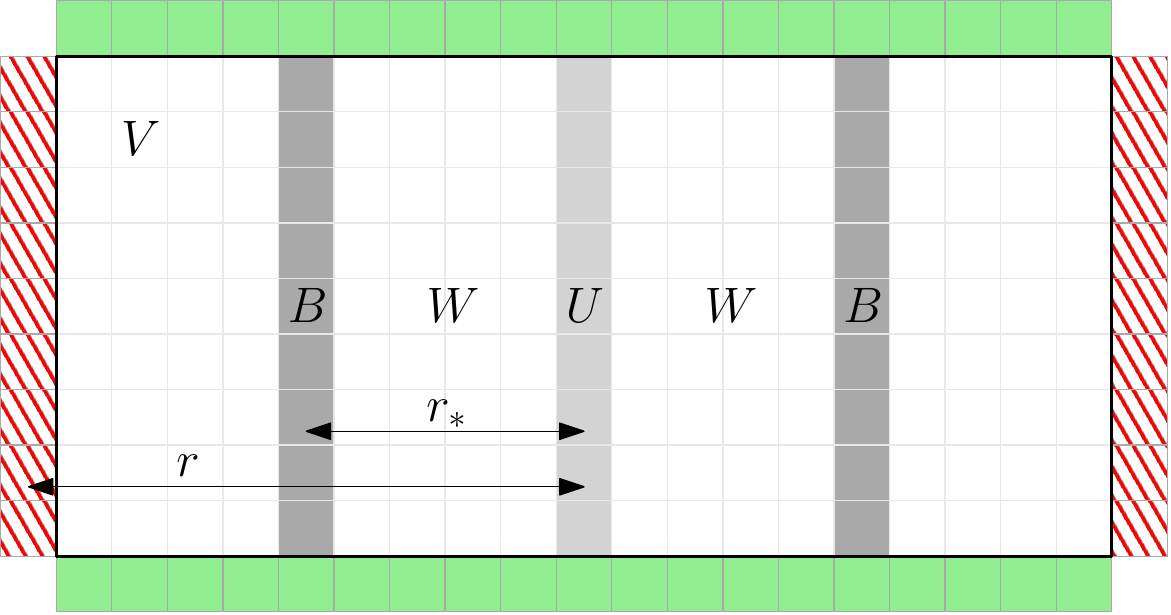}
		\caption{}
		\label{fig:ratio-mixing-box}		
	\end{subfigure}
	\caption{The set $\Sigma_{\tau,\tau'}$ of vertices in $\partial V$ where the two boundary conditions $\tau$ and~$\tau'$ disagree is depicted in dashed red, while the vertices of agreement are depicted in green. Strong spatial mixing allows one to couple samples from $P^\tau_V$ and $P^{\tau'}_V$ so that they agree on $B$ with high probability. Since $B$ is far from both $U$ and $\Sigma_{\tau,\tau'}$, using strong spatial mixing again, one may deduce that the densities of $P^\tau_{V,U}$ and $P^{\tau'}_{V,U}$ are similar as in~\eqref{eq:ratio-SSM-def}.}
	\label{fig:ratio-mixing}
\end{figure}

It remains to check that this coupling satisfies~\eqref{eq:H-prob}.
Note that, by construction and~\eqref{eq:B-W-props},
\[ \{\omega_B = \sigma_B\} \subset \{\omega_W = \sigma_W\} \subset \{\omega_U = \sigma_U\} .\]
This is the only use we make of the optimality of the coupling chosen in the second step of the construction.
Thus, it suffices to show that
\begin{equation}\label{eq:H-prob2}
\Pr\big(\omega_B=\sigma_B \mid \omega_U\big) \ge 1-\epsilon \qquad\text{almost surely} .
\end{equation}
Denote
\[ G := \Pr\big(\omega_B \neq \sigma_B \mid \omega_B\big) .\]
Denote $s := \sqrt{\rho(r_*)}$ and note that $s|B| \le \epsilon/3$, since $|B| \le s_{r_*}|U|$.
By Markov's inequality, the first step in the construction of the coupling, and strong spatial mixing,
\begin{equation}\label{eq:G-bound1}
\Pr(G \ge s) \le \frac{\E[G]}{s} = \frac{\Pr(\omega_B \neq \sigma_B)}{s} = \frac{\|P^\tau_{V,B}-P^{\tau'}_{V,B}\|}{s} \le \frac{|B| \cdot \rho(r_*)}{s} = s|B| \le \epsilon/3 .
\end{equation}
Using again strong spatial mixing, and since $s \le 1$, we similarly obtain, almost surely,
\begin{equation}\label{eq:G-bound2}
\big|\Pr(G \ge s \mid \omega_U) - \Pr(G \ge s)\big| \le \big\|\Pr(\omega_B \in \cdot \mid \omega_U) - \Pr(\omega_B \in \cdot)\big\| \le s^2|B| \le \epsilon/3 .
\end{equation}
Note that, by construction, $\omega_U$ and $\sigma_B$ are conditionally independent given $\omega_B$. Indeed, this follows from the fact that, given $(\omega_B,\sigma_B)$, the distribution of $\omega_W$ is almost surely $P^\tau_V(\omega_W \in \cdot \mid \omega_B)$. Thus, almost surely,
\[ \Pr\big(G \le s,~\omega_B \neq \sigma_B \mid \omega_U\big)
 \le \Pr\big(\omega_B \neq \sigma_B \mid G \le s,~\omega_U\big) = \E\big[G \mid G \le s,~\omega_U\big] \le s \le \epsilon/3 .\]
Putting this together with~\eqref{eq:G-bound1} and~\eqref{eq:G-bound2} yields~\eqref{eq:H-prob2}.
\end{proof}

\subsection{Existence of a contracting grand coupling}
\label{sec:grand-coupling}

Theorem~\ref{thm:grand-coupling-exists} follows by taking $r$ to be a large multiple of $\log n$ in the next proposition. Recall the definition of $\rho_*$ from~\eqref{eq:def-rho-star}.

\begin{prop}\label{prop:grand-coupling-exists}
	Let $P$ be a specification satisfying strong spatial mixing with rate function $\rho$.
	Then, for every $1 \le r \le n/4$, there exists a grand coupling $\pi_n$ of $(P^\tau_{\Lambda_n})_{\tau \in \Omega_{\partial \Lambda_n}}$ such that
	\begin{equation}\label{eq:grand-coupling}
	\lambda_{\pi_n}(\tau,A) \le \rho_*(r) + C_d r^d \cdot \frac{|A|}{|\Lambda_n|} \qquad\text{for all $\tau \in \Omega_{\partial \Lambda_n}$ and $A \subset \partial \Lambda_n$},
	\end{equation}
	where $C_d$ is a positive constant depending only on the dimension.
\end{prop}

\begin{proof}
Fix $1 \le r < s \le 4r \le n$ and denote $V := \Lambda_n$ and $U := \Lambda_{n-r}$.
In order to ease notation, we make the convention that whenever $\tau$ is used as an index with an unspecified domain, it runs over all elements in $\Omega_{\partial V}$.

Let $L_1,L_2,\dots,L_m \subset V \setminus U$ be disjoint (to be suitably chosen later).
The main idea in the construction of $\pi_n$ is to couple the configurations, in order, first on $U$, then on $L_1$, then on $L_2$, and so on, and finally on each connected component of $V \setminus (U \cup L_1 \cup \dots \cup L_m)$. See Figure~\ref{fig:coupling}.

Set $U_0 := U$ and $V_0 := V$, and for $1 \le i \le m$, denote
\[ U_i := U \cup L_1 \cup \cdots \cup L_i \qquad\text{and}\qquad V_i := V \setminus U_i .\]
For $1 \le i \le m$, denote
\[ T_i := \big\{ v \in \partial V_{i-1} : \dist(v,L_i) < s \big\} \]
and, for each $\eta \in \Omega_{T_i}$, denote
\[ \Omega_{\partial V_{i-1},\eta} := \big\{ \tau \in \Omega_{\partial V_{i-1}} : \tau_{T_i}=\eta \big \} . \]
Note that, for every $1 \le i \le m$,
\[ \{ \Omega_{\partial V_{i-1},\eta} \}_{\eta \in \Omega_{T_i}}\text{ is a partition of }\Omega_{\partial V_{i-1}} .\]
Let $Q$ be an optimal coupling of $(P^\tau_{V,U})_\tau$. For each $1 \le i \le m$ and $\eta \in \Omega_{T_i}$, let $Q_i^\eta$ be an optimal coupling of $(P^\tau_{V_i,L_i})_{\tau \in \Omega_{\partial V_{i-1},\eta}}$.
Let $Q_i$ be any coupling of $(P^\tau_{V_i,L_i})_{\tau \in \Omega_{\partial V_{i-1}}}$ with marginals $(Q_i^\eta)_{\eta \in \Omega_{T_i}}$, e.g., the product coupling.
Note that, by ratio strong (or weak) spatial mixing (which follows from the assumption and Proposition~\ref{prop:SSM-implies-ratio-SSM}),
\begin{equation}\label{eq:Q-coupling}
Q\left(\omega^\tau_U=\omega^{\tau'}_U\text{ for all }\tau,\tau'\right) \ge 1 - |U| \cdot \rho_*(r) ,
\end{equation}
and that, by ratio strong spatial mixing, for any $1 \le i \le m$ and $\eta \in \Omega_{T_i}$,
\begin{equation}\label{eq:Q_i-coupling}
Q_i\left(\omega^\tau_{L_i}=\omega^{\tau'}_{L_i}\text{ for all }\tau,\tau' \in \Omega_{\partial V_{i-1},\eta} \right) \ge 1 - |L_i| \cdot \rho_*(s) .
\end{equation}

The grand coupling $\pi_n$ is constructed as follows.
First, we sample $(\omega^\tau_U)_\tau$ from $Q$.
Next, we sample $(\omega^\tau_{L_i})_\tau$, for each $i=1,\dots,m$, one after another, as follows.
Let $i \in \{1,\dots,m\}$ and suppose we have already sampled $(\omega^\tau_{U_i})_\tau$.
We then sample $(\omega^\tau_{L_i})_\tau$ by setting $\omega^\tau_{L_i} := \sigma^{\omega^\tau_{U_i}}$ for all $\tau$, where $(\sigma^\tau)_\tau$ is sampled from $Q_i$, independently of anything else previously sampled.
Thus, at the end of this process, we have sampled $(\omega^\tau_{U_m})_\tau$.
To complete the construction, for each connected component $\cC$ of $V_m$, we independently sample $(\omega^\tau_\cC)_\tau$ from an optimal coupling of $(P^\tau_{V,\cC}(\cdot \mid \omega^\tau_{V_m}))_\tau$.
Define $\pi_n$ to be the distribution of the sample $(\omega^\tau)_\tau$ thus obtained.
It is straightforward to check that $\pi_n$ is a grand coupling of $(P^\tau_V)_\tau$.

\begin{figure}
	\centering
	\begin{subfigure}[t]{.45\textwidth}
		\centering
		\includegraphics[scale=1]{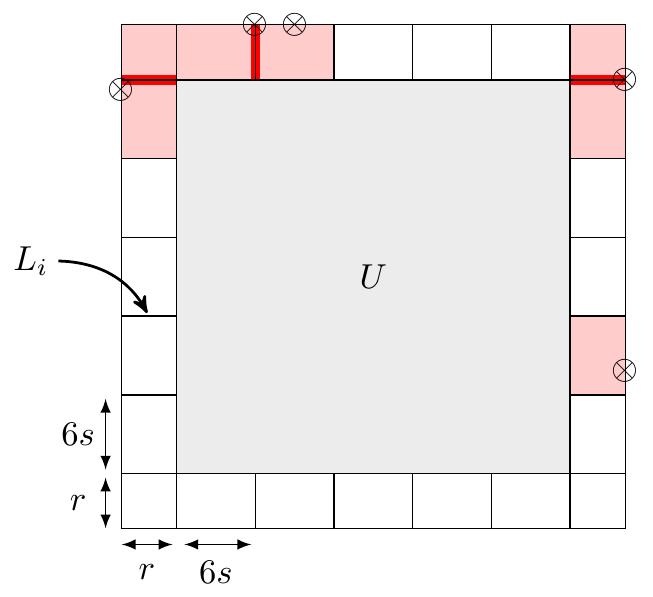}
		\caption{$d=2$.}
		\label{fig:coupling-a}
	\end{subfigure}%
	\begin{subfigure}{20pt}
		\quad
	\end{subfigure}%
	\begin{subfigure}[t]{.45\textwidth}
		\centering
		\includegraphics[scale=1]{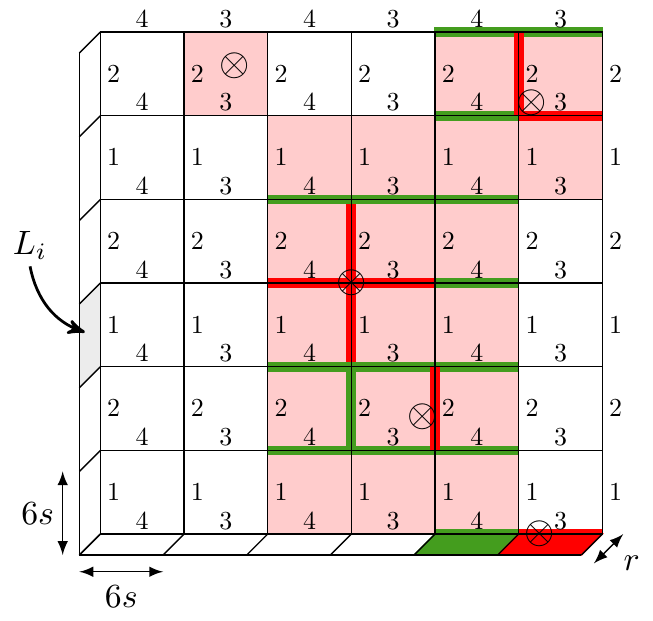}
		\caption{$d=3$; one of six slabs.}
		\label{fig:coupling-b}
	\end{subfigure}
	\caption{The set $V \setminus U$ is divided into small cubes having sides of length~$r$ or $s \le s' \le 10s$ (with $s'$ taken to be $6s$ in the figure).
		The faces of the cubes correspond to the $L_i$ (except for those faces which are contained in $U$ or $\partial V$) and the numbers adjacent to them indicate their relative order (the order is irrelevant when $d=2$).
		The elements of $A$ (the set of disagreement boundary vertices), along with circles of radius $s$ around them, are depicted by $\otimes$.
		We wish to couple samples $(\omega^\tau)_\tau$ from $(P^\tau_V)_\tau$, under all possible boundary conditions $\tau \in \Omega_{\partial V}$, in such a way that, for the set $\cT$ of boundary conditions which agree with a fixed configuration on $\partial V \setminus A$, with high probability, all samples in $(\omega^\tau)_{\tau \in \cT}$ coincide on the region of $V$ that is far away from $A$.
		Furthermore, we wish to do this in a manner which does not depend on $\cT$ or $A$.
		We do so by first coupling $(\omega^\tau)_\tau$ on $U$, then on $L_1$, then on $L_2$, and so forth, and finally on each of the remaining connected components.
		The $L_i$ for which $i \in I^A$ (i.e., for which $\dist(L_i,A) < s$) are directly affected by the disagreement on $A$ (these faces are depicted in red and are precisely those that have a circle $\otimes$ touching them), and we cannot hope to successfully couple all samples there.
		As this effect propagates, the $L_i$ for which $i \in \cI^A$ (i.e., for which $i' \to i$ for some $i' \in I^A$) but $i \notin I^A$ are also indirectly affected by the disagreement on $A$ (these faces are depicted in green).
		The connected components of $V_m$ that are contained in $\cC^A$ (i.e., components whose boundaries touch $A$ or such an affected $L_i$) are also directly or indirectly affected (these are depicted in pink). A well-chosen order of the $L_i$ ensures that the total effect of the disagreement on $A$ is not too large.}
	\label{fig:coupling}
\end{figure}

Let us observe some simple properties that follow from this construction.
First, by~\eqref{eq:Q-coupling},
\begin{equation}\label{eq:coupling-E-prob}
\pi_n(E) \ge 1 - |U| \cdot \rho_*(r), \qquad\text{where }E := \Big\{\omega^\tau_U=\omega^{\tau'}_U\text{ for all }\tau,\tau' \Big\} .
\end{equation}
For $\cT \subset \Omega_{\partial V}$, let $\Sigma_\cT$ be the subset of $\partial V$ on which some boundary conditions in $\cT$ do not agree, and let $E^\cT_i$ be the event that all boundary conditions in $\cT$ are coupled on $L_i$, i.e.,
\[ \Sigma_\cT := \Big\{ v \in \partial V : \tau_v \neq \tau'_v\text{ for some }\tau,\tau' \in \cT \Big\} ,\qquad E_i^\cT := \Big\{\omega^\tau_{L_i}=\omega^{\tau'}_{L_i}\text{ for all }\tau,\tau'\in\cT \Big\} .\]
By~\eqref{eq:Q_i-coupling}, for any $1 \le i \le m$ and $\cT \subset \Omega_{\partial V}$ such that $\dist(L_i,\Sigma_\cT) \ge s$,
\begin{equation}\label{eq:coupling-E_i-prob}
\pi_n\Big(E_i^\cT \mid (\omega^\tau_{U_i})_\tau\Big) \ge 1 - |L_i| \cdot \rho_*(s) \qquad\text{almost surely on the event }E \cap \bigcap_{\substack{1 \le i'<i\\\dist(L_i,L_{i'})<s}} E_{i'}^\cT .
\end{equation}
In addition, for any connected component $\cC$ of $V_m$,
\begin{equation}\label{eq:coupling-holes}
\omega^\tau_\cC = \omega^{\tau'}_\cC \qquad\text{almost surely for any }\tau,\tau'\text{ such that }\tau_{\partial \cC} = \tau'_{\partial \cC}.
\end{equation}
Equations~\eqref{eq:coupling-E-prob}, \eqref{eq:coupling-E_i-prob} and~\eqref{eq:coupling-holes} summarize the properties we need from the construction of $\pi_n$. 
We claim that, for a suitable choice of $L_1,\dots,L_m$, $\pi_n$ satisfies~\eqref{eq:grand-coupling}.
In fact, there are many possible such choices, and any choice with, say, property~\eqref{eq:small-effect} below will suffice.

For $v \in \partial V$, denote
\[ I^v := \big\{ 1 \le i \le m : \dist(v,L_i) < s \big\} .\]
Say that $i$ is a \emph{child} of $i'$ if $i>i'$ and $\dist(L_i,L_{i'}) < s$. Write $i' \to i$ if $i$ is a descendant (child, grandchild, etc.) of $i'$ or if $i=i'$.
Denote
 \[ \cI^v := \big\{ 1 \le i \le m : i' \to i\text{ for some }i' \in I^v \big\} .\]
We think of $I^v$ as indicating those $L_i$ which are directly affected by $v$, and $\cI^v \setminus I^v$ as those which are indirectly affected (as the effect of $v$ may propagate).
Denote
\[ \cL^v := \bigcup_{i \in \cI^v} L_i .\]
Let $\cC^v$ be the union of the connected components of $V_m$ which are adjacent to $\cL^v \cup \{v\}$. Then we shall show that the following suffices to ensure that $\pi_n$ satisfies~\eqref{eq:grand-coupling}:
\begin{equation}\label{eq:small-effect}
 |\cL^v \cup \cC^v| \le C_d r^d \qquad\text{for all }v \in \partial V .
\end{equation}
Intuitively, this property implies that the total negative effect caused by a disagreement at a single vertex $v \in \partial V$ cannot be too large, since~\eqref{eq:small-effect} provides a bound on the number of vertices affected (via the propagation of failed coupling events) by any such disagreement.

A choice of $L_1,\dots,L_m$ which satisfies this property can be obtained as follows (we give an informal description; see also Figure~\ref{fig:coupling}). In each of the $2d$ directions, there is a ``thick $(d-1)$-dimensional slab'' (of side-length $2(n-r)$ and of thickness~$r$) between the faces of $U$ and $V$ in that direction. We divide each such slab into thick $(d-1)$-dimensional cubes of side-length between $s$ and $10s$ (e.g., $[0,r] \times [0,s]^{d-1}$). We then continue to slabs of one less dimension: there are $\binom{2d}{2}$ ``thick $(d-2)$-dimensional slabs'' (of side-length $2(n-r)$ in $d-2$ of the dimensions and thickness $r$ in the two other dimensions), each of which is divided into thick $(d-2)$-dimensional cubes of side-length between $s$ and $10s$ (e.g., $[0,r]^2 \times [0,s]^{d-2}$). We continue this process to lower-dimensional slabs in a similar manner. At the end of the process, we obtain a partition of $V\setminus U$ into cubes of the form $[0,r]^k \times [0,s']^{d-k}$ with $1 \le k \le d$ and $s \le s' \le 10s$ (up to translations and coordinate permutations).
The $(d-1)$-dimensional faces of these cubes (where, for example, $[0,r]^k \times [0,s]^{d-k-1} \times \{0\}$ is such a face of $[0,r]^k \times [0,s]^{d-k}$) will be the $\{L_i\}_i$ (except for those faces which are contained in $U$ or $\partial V$, that is, we only consider faces which ``connect'' between $U$ and $\partial V$), chosen in a suitable order. Such an order is, for instance, obtained as follows.
Associate to each face the index $z \in \{1,\dots,2d\}$ of the zero-length coordinate and the set $J \subset \{1,\dots,2d\}$ of coordinates in which the face has side-length at least $s$ (e.g., $[0,s]^k \times [0,r]^{d-k-1} \times \{0\}$ has $z=2d$ and $J=\{1,\dots,k\}$). Note that the adjacency relation induced on the set of faces with a given $J$ and $z$ is that of (disjoint copies of) a $|J|$-dimensional cube. We let $p=(p_1,\dots,p_{|J|})$ denote the face's position in this $|J|$-dimensional grid (e.g., if $[0,s]^2 \times [0,r] \times \{0\}$ has $p=(0,0)$ then $[s,2s] \times [5s,6s] \times [0,r] \times \{0\}$ has $p=(1,5)$).
The faces are then ordered first in decreasing order according to $|J|$ (i.e., in increasing order according to the number of sides of length $r$), then in (say, increasing) order according to $z$ and then in lexicographical order according to $(p_1\text{ mod }2,\dots,p_{|J|}\text{ mod }2)$.
We complete this partial order to a linear order in an arbitrary manner.
This choice of ordering ensures that $i' \to i$ implies that the ``grid-distance'' between the faces $L_i$ and $L_{i'}$ (i.e., the $\ell_1$-distance between the $p$'s corresponding to the two faces) is at most $C'_d$.
In particular, since it is clear that $|I^v| \le C''_d$, we have $|\cI^v| \le C'''_d$. Since the faces have size at most $r(10s)^{d-2}$ and the connected components of $V_m$ have size at most $r(10s)^{d-1}$, we have $|\cL^v \cup \cC^v| \le C_d r^d$, as required.

The above construction is particularly simple (even somewhat degenerate) in two dimensions (see Remark~\ref{rem:two-dimensional-proof} below). However, writing it down precisely for $d \ge 3$ is a bit cumbersome and we omit the details.

We proceed to show that~\eqref{eq:small-effect} implies that $\pi_n$ satisfies~\eqref{eq:grand-coupling}. Towards showing this, let $(\omega^\tau)_\tau$ be sampled from $\pi_n$, let $\tau \in \Omega_{\partial V}$ and $A \subset \partial V$, and define
\[ \cT := \big\{ \tau' \in \Omega_{\partial V} : \tau_{\partial V \setminus A}=\tau'_{\partial V \setminus A} \big\} \qquad\text{and}\qquad \cW := \Big\{ v \in V : \omega^\tau_v=\omega^{\tau'}_v\text{ for all }\tau' \in \cT \Big\}.\]
We must show that
\[ \E|\cW| \ge |V| - |A| \cdot C_d r^d - |V| \cdot \rho_*(r) .\]
This will follow if we show that, for some $W \subset V$,
\begin{equation}\label{eq:W}
|W| \ge |V| - |A| \cdot C_d r^d \qquad\text{and}\qquad \Pr(W \subset \cW) \ge 1 - |V| \cdot \rho_*(r) .
\end{equation}
Denote $I^A := \bigcup_{v \in A} I^v$, $\cI^A := \bigcup_{v \in A} \cI^v$, $\cL^A := \bigcup_{v \in A} \cL^v$ and $\cC^A := \bigcup_{v \in A} \cC^v$.
Let $\cC_1,\dots,\cC_k$ be the connected components of $V_m$ which are not adjacent to $\cL^A \cup A$. Note that 
\[ \bigcup_{j=1}^k \cC_j = V_m \setminus \cC^A .\]
Set
\[ W := V \setminus (\cL^A \cup \cC^A) \]
and note that~\eqref{eq:small-effect} implies the first part of~\eqref{eq:W}.
By~\eqref{eq:coupling-holes}, we have
\[ E \cap \bigcap_{i \in I} E_i^\cT \subset \big\{ W \subset \cW \big\} ,\qquad\text{where }I := \{1,\dots,m\} \setminus \cI^A .\]
Since $\Pr(E^c) \le |U| \cdot \rho_*(r)$ by~\eqref{eq:coupling-E-prob}, the second part of~\eqref{eq:W} will follow if we show that, for any $i \in I$,
\[ \Pr\Big((E_i^\cT)^c \mid E\text{ occurs and }E_{i'}\text{ occurs for all }i' \in I\text{ with }i'<i\Big) \le |L_i| \cdot \rho_*(s) .\]
Indeed, this follows from~\eqref{eq:coupling-E_i-prob}, since $\dist(L_i,A) \ge s$ and $\Sigma_\cT \subset A$, and since, if $1 \le i'<i$ satisfies $\dist(L_i,L_{i'})<s$, then $i' \in I$ (otherwise $i'' \to i' \to i$ for some $i'' \in I^A$, implying that $i \in \cI^A$).
\end{proof}

\begin{remark}\label{rem:two-dimensional-proof}
	In the two-dimensional case, the proof of Proposition~\ref{prop:grand-coupling-exists} is slightly simplified due to the fact that the constructed collection $\{L_i\}_i$ satisfies $\dist(L_i,L_{i'}) \ge s$ whenever $i \neq i'$, except when $L_i$ and $L_{i'}$ are both at one of the same four corners (see Figure~\ref{fig:coupling-a}). In fact, one may just as well remove the eight $L_i$ which are at the corners so that this holds for all $i \neq i'$. Thus, regardless of the order in which $(L_i)_i$ is chosen, we have $i' \to i$ if and only if $i'=i$, and the event on the right-hand side of~\eqref{eq:coupling-E_i-prob} is just $E$. On the other hand, although these are convenient simplifications, they are not of significant importance. Instead, what is important in two dimensions is that in order to deduce~\eqref{eq:Q_i-coupling}, we only need exponential strong spatial mixing for a restricted class of sets $(U,V)$ and boundary conditions $(\tau,\tau')$, that is, we only need to assume that~\eqref{eq:SSM-def} holds for certain $(U,V)$ and $(\tau,\tau')$. Indeed, the proof of Proposition~\ref{prop:SSM-implies-ratio-SSM} reveals that if one only assumes exponential strong spatial mixing for $(U,V)=(\{0\}\times S_r,S_r \times S_r)$ and $(\tau,\tau')$ for which $\Sigma_{\tau,\tau'} \subset \{-r-1,r+1\}\times S_r$, where $S_r := \{-r,\dots,r\}$, then one obtains as a consequence exponential ratio strong spatial mixing for $(U,V)=(\{0\}\times S_r,S_{2r+1} \times S_r)$ and $(\tau,\tau')$ for which $\Sigma_{\tau,\tau'} \subset \{-2r-2,2r+2\}\times S_r$ (see Figure~\ref{fig:ratio-mixing-box}).	
	Finally, this is all that is needed to obtain~\eqref{eq:Q_i-coupling}, so that the conclusion of Proposition~\ref{prop:grand-coupling-exists} remains true under this weaker hypothesis.
\end{remark}

\section{An obstruction for \ffiid}
\label{sec:no-ffiid}

In this section, we prove Theorem~\ref{thm:no-ffiid}.

\begin{proof}[Proof of Theorem~\ref{thm:no-ffiid}]
Let $X$ and $X'$ be as in the theorem.
Denote the distributions of $X$ and $X'$ by $\mu$ and $\nu$.
We write $\omega$ for a generic random element of $\Omega$.
Let $A$ be a cylinder event satisfying $\mu(A) < \nu(A)$, which is determined by the values of $\omega$ on a box $\Lambda_k$.
For $n>k$, denote
\[ Z_n := \frac{1}{|\Lambda_{n-k}|} \sum_{v \in \Lambda_{n-k}} \1_A\big((\omega_{u+v})_{u \in \Z^d}\big) .\]
Note that $Z_n$ is measurable with respect to $\omega_{\Lambda_n}$.
By the pointwise ergodic theorem,
\begin{equation}\label{eq:A-has-prob-1}
 \mu\big(Z_n \to \mu(A)\text{ as }n \to \infty\big)=1 \qquad\text{and}\qquad \nu\big(Z_n \to \nu(A)\text{ as }n \to \infty\big)=1 .
\end{equation}
Assume towards a contradiction that $X$ is \ffiid.
Then the random field $W=(W_v)_{v \in \Z^d}$ defined by $W_v :=\1_A\big((X_{u+v})_{u \in \Z^d}\big)$ is also \ffiid, so that the convergence in the ergodic theorem occurs at an exponential rate for $W$~\cite{bosco2010exponential} (see also~\cite{van1999existence,marton1994}). Hence, denoting $a := \frac12(\mu(A)+\nu(A))$,
\[ \mu(Z_n \ge a) \le Ce^{-2cn^d} \qquad\text{for some }C,c>0\text{ and for all }n>k .\]
By Markov's inequality,
\[ \Pr\big(X_{\partial \Lambda_n} \in T_n\big) \le Ce^{-cn^d}, \qquad\text{where }T_n := \Big\{ \tau \in \Omega_{\partial \Lambda_n} : P^\tau_{\Lambda_n}(Z_n \ge a) \ge e^{-cn^d} \Big\} .\]
Thus, since $\Pr\big(X'_{\partial \Lambda_n} \in T_n\big)=1-o(1)$ by~\eqref{eq:A-has-prob-1}, we have reached a contradiction to~\eqref{eq:no-ffiid-assumption}.
\end{proof}

\bibliographystyle{amsplain}
%\nocite{*}
\bibliography{library}

% \bib, bibdiv, biblist are defined by the amsrefs package.
\begin{bibdiv}
\begin{biblist}

\bib{achlioptas2005rapid}{article}{
      author={Achlioptas, Dimitris},
      author={Molloy, Mike},
      author={Moore, Cristopher},
      author={Van~Bussel, Frank},
       title={Rapid mixing for lattice colourings with fewer colours},
        date={2005},
     journal={Journal of Statistical Mechanics: Theory and Experiment},
      volume={2005},
      number={10},
       pages={P10012},
}

\bib{aizenman1986critical}{article}{
      author={Aizenman, M},
      author={Fern{\'a}ndez, R},
       title={On the critical behavior of the magnetization in high-dimensional
  {I}sing models},
        date={1986},
     journal={Journal of Statistical Physics},
      volume={44},
      number={3-4},
       pages={393\ndash 454},
}

\bib{aizenman2015random}{article}{
      author={Aizenman, Michael},
      author={Duminil-Copin, Hugo},
      author={Sidoravicius, Vladas},
       title={Random currents and continuity of {I}sing model’s spontaneous
  magnetization},
        date={2015},
     journal={Communications in Mathematical Physics},
      volume={334},
      number={2},
       pages={719\ndash 742},
}

\bib{akcoglu1979finitary}{article}{
      author={Akcoglu, M~A},
      author={del Junco, Andr{\'e}s},
      author={Rahe, M},
       title={Finitary codes between {M}arkov processes},
        date={1979},
     journal={Zeitschrift f{\"u}r Wahrscheinlichkeitstheorie und Verwandte
  Gebiete},
      volume={47},
      number={3},
       pages={305\ndash 314},
}

\bib{alexander1998weak}{article}{
      author={Alexander, Kenneth~S},
       title={On weak mixing in lattice models},
        date={1998},
     journal={Probability theory and related fields},
      volume={110},
      number={4},
       pages={441\ndash 471},
}

\bib{alexander2004mixing}{article}{
      author={Alexander, Kenneth~S},
       title={Mixing properties and exponential decay for lattice systems in
  finite volumes},
        date={2004},
     journal={The Annals of Probability},
      volume={32},
      number={1A},
       pages={441\ndash 487},
}

\bib{angel2012stationary}{inproceedings}{
      author={Angel, Omer},
      author={Benjamini, Itai},
      author={Gurel-Gurevich, Ori},
      author={Meyerovitch, Tom},
      author={Peled, Ron},
       title={Stationary map coloring},
organization={Institut Henri Poincar{\'e}},
        date={2012},
   booktitle={Annales de l'institut henri poincar{\'e}, probabilit{\'e}s et
  statistiques},
      volume={48},
       pages={327\ndash 342},
}

\bib{angel2019pairwise}{article}{
      author={Angel, Omer},
      author={Spinka, Yinon},
       title={Pairwise optimal coupling of multiple random variables},
        date={2019},
     journal={arXiv preprint arXiv:1903.00632},
}

\bib{van1999absence}{incollection}{
      author={Berg, Jacob van~den},
       title={On the absence of phase transition in the monomer-dimer model},
        date={1999},
   booktitle={Perplexing problems in probability},
   publisher={Springer},
       pages={185\ndash 195},
}

\bib{van1994disagreement}{article}{
      author={Berg, Jacob van~den},
      author={Maes, Christian},
       title={Disagreement percolation in the study of {M}arkov fields},
        date={1994},
     journal={The Annals of Probability},
      volume={22},
      number={2},
       pages={749\ndash 763},
}

\bib{van1994percolation}{article}{
      author={Berg, Jacob van~den},
      author={Steif, Jeffrey~E},
       title={Percolation and the hard-core lattice gas model},
        date={1994},
     journal={Stochastic Processes and their Applications},
      volume={49},
      number={2},
       pages={179\ndash 197},
}

\bib{van1999existence}{article}{
      author={Berg, Jacob van~den},
      author={Steif, Jeffrey~E},
       title={On the existence and nonexistence of finitary codings for a class
  of random fields},
        date={1999},
     journal={Annals of probability},
       pages={1501\ndash 1522},
}

\bib{blanca2018spatial}{inproceedings}{
      author={Blanca, Antonio},
      author={Caputo, Pietro},
      author={Sinclair, Alistair},
      author={Vigoda, Eric},
       title={Spatial mixing and non-local {M}arkov chains},
organization={SIAM},
        date={2018},
   booktitle={Proceedings of the twenty-ninth annual {ACM-SIAM} symposium on
  discrete algorithms},
       pages={1965\ndash 1980},
}

\bib{bosco2010exponential}{article}{
      author={Bosco, GG},
      author={Machado, FP},
      author={Ritchie, Thomas~L},
       title={Exponential rates of convergence in the ergodic theorem: a
  constructive approach},
        date={2010},
     journal={Journal of Statistical Physics},
      volume={139},
      number={3},
       pages={367\ndash 374},
}

\bib{briceno2017strong}{article}{
      author={Brice\~no, Raimundo},
      author={Pavlov, Ronnie},
       title={Strong spatial mixing in homomorphism spaces},
        date={2017},
     journal={SIAM Journal on Discrete Mathematics},
      volume={31},
      number={3},
       pages={2110\ndash 2137},
}

\bib{briceno2017topological}{article}{
      author={Briceno, Raimundo},
       title={The topological strong spatial mixing property and new conditions
  for pressure approximation},
        date={2017},
     journal={Ergodic Theory and Dynamical Systems},
       pages={1\ndash 39},
}

\bib{brightwell1999nonmonotonic}{article}{
      author={Brightwell, Graham~R},
      author={H{\"a}ggstr{\"o}m, Olle},
      author={Winkler, Peter},
       title={Nonmonotonic behavior in hard-core and {W}idom--{R}owlinson
  models},
        date={1999},
     journal={Journal of statistical physics},
      volume={94},
      number={3-4},
       pages={415\ndash 435},
}

\bib{brightwell2000gibbs}{article}{
      author={Brightwell, Graham~R},
      author={Winkler, Peter},
       title={Gibbs measures and dismantlable graphs},
        date={2000},
     journal={Journal of combinatorial theory, series B},
      volume={78},
      number={1},
       pages={141\ndash 166},
}

\bib{burton1994non}{article}{
      author={Burton, Robert},
      author={Steif, Jeffrey~E},
       title={Non-uniqueness of measures of maximal entropy for subshifts of
  finite type},
        date={1994},
     journal={Ergodic Theory and Dynamical Systems},
      volume={14},
      number={2},
       pages={213\ndash 235},
}

\bib{burton1995new}{article}{
      author={Burton, Robert},
      author={Steif, Jeffrey~E},
       title={New results on measures of maximal entropy},
        date={1995},
     journal={Israel Journal of Mathematics},
      volume={89},
      number={1-3},
       pages={275\ndash 300},
}

\bib{buvsic2013probabilistic}{article}{
      author={Bu{\v{s}}i{\'c}, Ana},
      author={Mairesse, Jean},
      author={Marcovici, Ir{\`e}ne},
       title={Probabilistic cellular automata, invariant measures, and perfect
  sampling},
        date={2013},
     journal={Advances in Applied Probability},
      volume={45},
      number={4},
       pages={960\ndash 980},
}

\bib{de2012developments}{article}{
      author={De~Santis, Emilio},
      author={Lissandrelli, Andrea},
       title={Developments in perfect simulation of {G}ibbs measures through a
  new result for the extinction of {G}alton-{W}atson-like processes},
        date={2012},
     journal={Journal of Statistical Physics},
      volume={147},
      number={2},
       pages={231\ndash 251},
}

\bib{de2008exact}{article}{
      author={De~Santis, Emilio},
      author={Piccioni, Mauro},
       title={Exact simulation for discrete time spin systems and unilateral
  fields},
        date={2008},
     journal={Methodology and Computing in Applied Probability},
      volume={10},
      number={1},
       pages={105\ndash 120},
}

\bib{junco1980finitary}{article}{
      author={del Junco, Andr{\'e}s},
       title={Finitary coding of {M}arkov random fields},
        date={1980},
     journal={Probability Theory and Related Fields},
      volume={52},
      number={2},
       pages={193\ndash 202},
}

\bib{del1990bernoulli}{article}{
      author={del Junco, Andr{\'e}s},
       title={Bernoulli shifts of the same entropy are finitarily and
  unilaterally isomorphic},
        date={1990},
     journal={Ergodic Theory and Dynamical Systems},
      volume={10},
      number={4},
       pages={687\ndash 715},
}

\bib{den2006random}{article}{
      author={den Hollander, Frank},
      author={Steif, Jeffrey~E},
       title={Random walk in random scenery: a survey of some recent results},
        date={2006},
     journal={Lecture Notes-Monograph Series},
       pages={53\ndash 65},
}

\bib{diaconis1999iterated}{article}{
      author={Diaconis, Persi},
      author={Freedman, David},
       title={Iterated random functions},
        date={1999},
     journal={SIAM review},
      volume={41},
      number={1},
       pages={45\ndash 76},
}

\bib{dimakos2001guide}{article}{
      author={Dimakos, Xeni~K},
       title={A guide to exact simulation},
        date={2001},
     journal={International statistical review},
      volume={69},
      number={1},
       pages={27\ndash 48},
}

\bib{Dobrushin1968TheDe}{article}{
      author={Dobrushin, Roland~L},
       title={The description of a random field by means of conditional
  probabilities and conditions of its regularity},
        date={1968},
     journal={Theor. Probab. Appl.},
      volume={13},
       pages={197\ndash 224},
}

\bib{dobrushin1985completely}{incollection}{
      author={Dobrushin, Roland~L},
      author={Shlosman, Senya~B},
       title={Completely analytical {G}ibbs fields},
        date={1985},
   booktitle={Statistical physics and dynamical systems},
   publisher={Springer},
       pages={371\ndash 403},
}

\bib{dobrushin1985constructive}{incollection}{
      author={Dobrushin, Roland~L},
      author={Shlosman, Senya~B},
       title={Constructive criterion for the uniqueness of {G}ibbs field},
        date={1985},
   booktitle={Statistical physics and dynamical systems},
   publisher={Springer},
       pages={347\ndash 370},
}

\bib{duminil2016discontinuity}{article}{
      author={Duminil-Copin, Hugo},
      author={Gagnebin, Maxime},
      author={Harel, Matan},
      author={Manolescu, Ioan},
      author={Tassion, Vincent},
       title={Discontinuity of the phase transition for the planar
  random-cluster and {P}otts models with $q>4$},
        date={2016},
     journal={arXiv preprint arXiv:1611.09877},
}

\bib{duminil2017sharp}{article}{
      author={Duminil-Copin, Hugo},
      author={Raoufi, Aran},
      author={Tassion, Vincent},
       title={Sharp phase transition for the random-cluster and {P}otts models
  via decision trees},
        date={2017},
     journal={arXiv preprint arXiv:1705.03104},
}

\bib{duminil2017continuity}{article}{
      author={Duminil-Copin, Hugo},
      author={Sidoravicius, Vladas},
      author={Tassion, Vincent},
       title={Continuity of the phase transition for planar random-cluster and
  {P}otts models with $1 \le q \le 4$},
        date={2017},
     journal={Communications in Mathematical Physics},
      volume={349},
      number={1},
       pages={47\ndash 107},
}

\bib{dyer2004mixing}{article}{
      author={Dyer, Martin},
      author={Sinclair, Alistair},
      author={Vigoda, Eric},
      author={Weitz, Dror},
       title={Mixing in time and space for lattice spin systems: a
  combinatorial view},
        date={2004},
     journal={Random Structures \& Algorithms},
      volume={24},
      number={4},
       pages={461\ndash 479},
}

\bib{edwards1988generalization}{article}{
      author={Edwards, Robert~G},
      author={Sokal, Alan~D},
       title={Generalization of the {F}ortuin--{K}asteleyn--{S}wendsen--{W}ang
  representation and {M}onte {C}arlo algorithm},
        date={1988},
     journal={Physical review D},
      volume={38},
      number={6},
       pages={2009},
}

\bib{fiebig1984return}{article}{
      author={Fiebig, U~R},
       title={A return time invariant for finitary isomorphisms},
        date={1984},
     journal={Ergodic Theory and Dynamical Systems},
      volume={4},
      number={2},
       pages={225\ndash 231},
}

\bib{fill2001stochastic}{article}{
      author={Fill, James~A},
      author={Machida, Motoya},
       title={Stochastic monotonicity and realizable monotonicity},
        date={2001},
     journal={Annals of probability},
       pages={938\ndash 978},
}

\bib{galves2008perfect}{article}{
      author={Galves, A},
      author={Garcia, NL},
      author={L{\"o}cherbach, E},
       title={Perfect simulation and finitary coding for multicolor systems
  with interactions of infinite range},
        date={2008},
     journal={arXiv preprint arXiv:0809.3494},
}

\bib{galves2010perfect}{article}{
      author={Galves, Antonio},
      author={L{\"o}cherbach, Eva},
      author={Orlandi, Enza},
       title={Perfect simulation of infinite range {G}ibbs measures and
  coupling with their finite range approximations},
        date={2010},
     journal={Journal of Statistical Physics},
      volume={138},
      number={1-3},
       pages={476\ndash 495},
}

\bib{gamarnik2015strong}{article}{
      author={Gamarnik, David},
      author={Katz, Dmitriy},
      author={Misra, Sidhant},
       title={Strong spatial mixing of list coloring of graphs},
        date={2015},
     journal={Random Structures \& Algorithms},
      volume={46},
      number={4},
       pages={599\ndash 613},
}

\bib{georgii2001random}{article}{
      author={Georgii, Hans-Otto},
      author={H{\"a}ggstr{\"o}m, Olle},
      author={Maes, Christian},
       title={The random geometry of equilibrium phases},
        date={2001},
     journal={Phase transitions and critical phenomena},
      volume={18},
       pages={1\ndash 142},
}

\bib{goldberg2006improved}{article}{
      author={Goldberg, Leslie~A},
      author={Jalsenius, Markus},
      author={Martin, Russell},
      author={Paterson, Mike},
       title={Improved mixing bounds for the anti-ferromagnetic {P}otts model
  on {${\bf Z}^2$}},
        date={2006},
     journal={LMS Journal of Computation and Mathematics},
      volume={9},
       pages={1\ndash 20},
}

\bib{goldberg2005strong}{article}{
      author={Goldberg, Leslie~A},
      author={Martin, Russell},
      author={Paterson, Mike},
       title={Strong spatial mixing with fewer colors for lattice graphs},
        date={2005},
     journal={SIAM Journal on Computing},
      volume={35},
      number={2},
       pages={486\ndash 517},
}

\bib{grimmett2006random}{book}{
      author={Grimmett, Geoffrey~R},
       title={The random-cluster model},
   publisher={Springer Science \& Business Media},
        date={2006},
      volume={333},
}

\bib{haggstrom1996phase}{article}{
      author={H{\"a}ggstr{\"o}m, Olle},
       title={On phase transitions for subshifts of finite type},
        date={1996},
     journal={Israel Journal of Mathematics},
      volume={94},
      number={1},
       pages={319\ndash 352},
}

\bib{haggstrom2002monotonicity}{article}{
      author={H{\"a}ggstr{\"o}m, Olle},
       title={A monotonicity result for hard-core and {W}idom-{R}owlinson
  models on certain $d$-dimensional lattices},
        date={2002},
     journal={Electronic Communications in Probability},
      volume={7},
       pages={67\ndash 78},
}

\bib{haggstrom1998exact}{article}{
      author={H{\"a}ggstr{\"o}m, Olle},
      author={Nelander, Karin},
       title={Exact sampling from anti-monotone systems},
        date={1998},
     journal={Statistica Neerlandica},
      volume={52},
      number={3},
       pages={360\ndash 380},
}

\bib{haggstrom1999exact}{article}{
      author={H{\"a}ggstr{\"o}m, Olle},
      author={Nelander, Karin},
       title={On exact simulation of {M}arkov random fields using coupling from
  the past},
        date={1999},
     journal={Scandinavian Journal of Statistics},
      volume={26},
      number={3},
       pages={395\ndash 411},
}

\bib{haggstrom2000propp}{article}{
      author={H{\"a}ggstr{\"o}m, Olle},
      author={Steif, Jeffrey~E},
       title={Propp--{W}ilson algorithms and finitary codings for high noise
  {M}arkov random fields},
        date={2000},
     journal={Combinatorics, Probability and Computing},
      volume={9},
      number={5},
       pages={425\ndash 439},
}

\bib{harel2018finitary}{article}{
      author={Harel, Matan},
      author={Spinka, Yinon},
       title={Finitary codings for the random-cluster model and other
  infinite-range monotone models},
        date={2018},
     journal={arXiv preprint arXiv:1808.02333},
}

\bib{harvey2006universal}{article}{
      author={Harvey, Nate},
      author={Holroyd, Alexander~E},
      author={Peres, Yuval},
      author={Romik, Dan},
       title={Universal finitary codes with exponential tails},
        date={2006},
     journal={Proceedings of the London Mathematical Society},
      volume={94},
      number={2},
       pages={475\ndash 496},
}

\bib{harvey2011invariant}{article}{
      author={Harvey, Nate},
      author={Peres, Yuval},
       title={An invariant of finitary codes with finite expected square root
  coding length},
        date={2011},
     journal={Ergodic Theory and Dynamical Systems},
      volume={31},
      number={1},
       pages={77\ndash 90},
}

\bib{holroyd2017one}{inproceedings}{
      author={Holroyd, Alexander~E},
       title={One-dependent coloring by finitary factors},
organization={Institut Henri Poincar{\'e}},
        date={2017},
   booktitle={Annales de l'institut henri poincar{\'e}, probabilit{\'e}s et
  statistiques},
      volume={53},
       pages={753\ndash 765},
}

\bib{holroyd2017mallows}{article}{
      author={Holroyd, Alexander~E},
      author={Hutchcroft, Tom},
      author={Levy, Avi},
       title={Mallows permutations and finite dependence},
        date={2017},
     journal={arXiv preprint arXiv:1706.09526},
}

\bib{holroyd2017finitary}{article}{
      author={Holroyd, Alexander~E},
      author={Schramm, Oded},
      author={Wilson, David~B},
       title={Finitary coloring},
        date={2017},
     journal={The Annals of Probability},
      volume={45},
      number={5},
       pages={2867\ndash 2898},
}

\bib{huber1998exact}{inproceedings}{
      author={Huber, Mark},
       title={Exact sampling and approximate counting techniques},
organization={ACM},
        date={1998},
   booktitle={Proceedings of the thirtieth annual acm symposium on theory of
  computing},
       pages={31\ndash 40},
}

\bib{huber2004perfect}{article}{
      author={Huber, Mark},
       title={Perfect sampling using bounding chains},
        date={2004},
     journal={The Annals of Applied Probability},
      volume={14},
      number={2},
       pages={734\ndash 753},
}

\bib{jalsenius2009strong}{article}{
      author={Jalsenius, Markus},
       title={Strong spatial mixing and rapid mixing with five colours for the
  {K}agome lattice},
        date={2009},
     journal={LMS Journal of Computation and Mathematics},
      volume={12},
       pages={195\ndash 227},
}

\bib{katznelson1972commuting}{article}{
      author={Katznelson, Yitzhak},
      author={Weiss, Benjamin},
       title={Commuting measure-preserving transformations},
        date={1972},
     journal={Israel Journal of Mathematics},
      volume={12},
      number={2},
       pages={161\ndash 173},
}

\bib{keane1977class}{article}{
      author={Keane, Michael},
      author={Smorodinsky, Meir},
       title={A class of finitary codes},
        date={1977},
     journal={Israel Journal of Mathematics},
      volume={26},
      number={3-4},
       pages={352\ndash 371},
}

\bib{keane1979bernoulli}{article}{
      author={Keane, Michael},
      author={Smorodinsky, Meir},
       title={Bernoulli schemes of the same entropy are finitarily isomorphic},
        date={1979},
     journal={Annals of Mathematics},
      volume={109},
      number={2},
       pages={397\ndash 406},
}

\bib{keane1979finitary}{article}{
      author={Keane, Michael},
      author={Smorodinsky, Meir},
       title={Finitary isomorphisms of irreducible {M}arkov shifts},
        date={1979},
     journal={Israel Journal of Mathematics},
      volume={34},
      number={4},
       pages={281\ndash 286},
}

\bib{keane2003finitary}{article}{
      author={Keane, Michael},
      author={Steif, Jeffrey~E},
       title={Finitary coding for the one-dimensional {${T,T^{-1}}$} process
  with drift},
        date={2003},
     journal={Annals of probability},
       pages={1979\ndash 1985},
}

\bib{kendall2000perfect}{article}{
      author={Kendall, Wilfrid~S},
      author={M{\o}ller, Jesper},
       title={Perfect simulation using dominating processes on ordered spaces,
  with application to locally stable point processes},
        date={2000},
     journal={Advances in Applied Probability},
      volume={32},
      number={3},
       pages={844\ndash 865},
}

\bib{krieger1983finitary}{article}{
      author={Krieger, Wolfgang},
       title={On the finitary isomorphisms of {M}arkov shifts that have finite
  expected coding time},
        date={1983},
     journal={Probability Theory and Related Fields},
      volume={65},
      number={2},
       pages={323\ndash 328},
}

\bib{lebowitz1971phase}{article}{
      author={Lebowitz, JL},
      author={Gallavotti, G},
       title={Phase transitions in binary lattice gases},
        date={1971},
     journal={Journal of Mathematical Physics},
      volume={12},
      number={7},
       pages={1129\ndash 1133},
}

\bib{li2013correlation}{inproceedings}{
      author={Li, Liang},
      author={Lu, Pinyan},
      author={Yin, Yitong},
       title={Correlation decay up to uniqueness in spin systems},
organization={SIAM},
        date={2013},
   booktitle={Proceedings of the twenty-fourth annual {ACM-SIAM} symposium on
  discrete algorithms},
       pages={67\ndash 84},
}

\bib{lubetzky2012critical}{article}{
      author={Lubetzky, Eyal},
      author={Sly, Allan},
       title={Critical {I}sing on the square lattice mixes in polynomial time},
        date={2012},
     journal={Communications in Mathematical Physics},
      volume={313},
      number={3},
       pages={815\ndash 836},
}

\bib{lubetzky2013cutoff}{article}{
      author={Lubetzky, Eyal},
      author={Sly, Allan},
       title={Cutoff for the {I}sing model on the lattice},
        date={2013},
     journal={Inventiones mathematicae},
      volume={191},
      number={3},
       pages={719\ndash 755},
}

\bib{lubetzky2014cutoff}{article}{
      author={Lubetzky, Eyal},
      author={Sly, Allan},
       title={Cutoff for general spin systems with arbitrary boundary
  conditions},
        date={2014},
     journal={Communications on Pure and Applied Mathematics},
      volume={67},
      number={6},
       pages={982\ndash 1027},
}

\bib{lyons2011perfect}{article}{
      author={Lyons, Russell},
      author={Nazarov, Fedor},
       title={Perfect matchings as {IID} factors on non-amenable groups},
        date={2011},
     journal={European Journal of Combinatorics},
      volume={32},
      number={7},
       pages={1115\ndash 1125},
}

\bib{marcus2013computing}{article}{
      author={Marcus, Brian},
      author={Pavlov, Ronnie},
       title={Computing bounds for entropy of stationary {${\bf Z}^d$} {M}arkov
  random fields},
        date={2013},
     journal={SIAM Journal on Discrete Mathematics},
      volume={27},
      number={3},
       pages={1544\ndash 1558},
}

\bib{martinelli1999lectures}{incollection}{
      author={Martinelli, Fabio},
       title={Lectures on {G}lauber dynamics for discrete spin models},
        date={1999},
   booktitle={Lectures on probability theory and statistics},
   publisher={Springer},
       pages={93\ndash 191},
}

\bib{martinelli1994approach}{article}{
      author={Martinelli, Fabio},
      author={Olivieri, Enzo},
       title={Approach to equilibrium of {G}lauber dynamics in the one phase
  region. {I}. {T}he attractive case},
        date={1994},
      volume={161},
      number={3},
       pages={447\ndash 486},
}

\bib{martinelli1994approach2}{article}{
      author={Martinelli, Fabio},
      author={Olivieri, Enzo},
       title={Approach to equilibrium of {G}lauber dynamics in the one phase
  region. {II}. {T}he general case},
        date={1994},
      volume={161},
      number={3},
       pages={487\ndash 514},
}

\bib{martinelli19942}{article}{
      author={Martinelli, Fabio},
      author={Olivieri, Enzo},
      author={Schonmann, Roberto~H},
       title={For 2-{D} lattice spin systems weak mixing implies strong
  mixing},
        date={1994},
     journal={Communications in Mathematical Physics},
      volume={165},
      number={1},
       pages={33\ndash 47},
}

\bib{marton1994}{article}{
      author={Marton, Katalin},
      author={Shields, Paul~C},
       title={The positive-divergence and blowing-up properties},
        date={1994},
     journal={Israel Journal of Mathematics},
      volume={86},
      number={1},
       pages={331\ndash 348},
}

\bib{meshalkin1959case}{article}{
      author={Meshalkin, LD},
       title={A case of isomorphism of {B}ernoulli schemes},
        date={1959},
     journal={DOKLADY AKADEMII NAUK SSSR},
      volume={128},
      number={1},
       pages={41\ndash 44},
}

\bib{mossel2008rapid}{inproceedings}{
      author={Mossel, Elchanan},
      author={Sly, Allan},
       title={Rapid mixing of {G}ibbs sampling on graphs that are sparse on
  average},
organization={Society for Industrial and Applied Mathematics},
        date={2008},
   booktitle={Proceedings of the nineteenth annual {ACM-SIAM} symposium on
  discrete algorithms},
       pages={238\ndash 247},
}

\bib{nair2007correlation}{article}{
      author={Nair, Chandra},
      author={Tetali, Prasad},
       title={The correlation decay ({CD}) tree and strong spatial mixing in
  multi-spin systems},
        date={2007},
     journal={arXiv preprint math},
        ISSN={0701494/},
}

\bib{ornstein1970bernoulli}{article}{
      author={Ornstein, Donald},
       title={Bernoulli shifts with the same entropy are isomorphic},
        date={1970},
     journal={Advances in Mathematics},
      volume={4},
      number={3},
       pages={337\ndash 352},
}

\bib{peled2014odd}{inproceedings}{
      author={Peled, Ron},
      author={Samotij, Wojciech},
       title={Odd cutsets and the hard-core model on {${\bf Z}^d$}},
organization={Institut Henri Poincar{\'e}},
        date={2014},
   booktitle={Annales de l'institut henri poincar{\'e}, probabilit{\'e}s et
  statistiques},
      volume={50},
       pages={975\ndash 998},
}

\bib{peled2017condition}{article}{
      author={Peled, Ron},
      author={Spinka, Yinon},
       title={A condition for long-range order in discrete spin systems with
  application to the antiferromagnetic {P}otts model},
        date={2017},
     journal={arXiv preprint arXiv:1712.03699},
}

\bib{peled2018rigidity}{article}{
      author={Peled, Ron},
      author={Spinka, Yinon},
       title={Rigidity of proper colorings of {${\mathbb Z}^d$}},
        date={2018},
     journal={arXiv preprint arXiv:1808.03597},
}

\bib{propp1998coupling}{article}{
      author={Propp, James},
      author={Wilson, David},
       title={Coupling from the past: a user’s guide},
        date={1998},
     journal={Microsurveys in Discrete Probability},
      volume={41},
       pages={181\ndash 192},
}

\bib{propp1996exact}{article}{
      author={Propp, James~G},
      author={Wilson, David~B},
       title={Exact sampling with coupled {M}arkov chains and applications to
  statistical mechanics},
        date={1996},
     journal={Random structures and Algorithms},
      volume={9},
      number={1-2},
       pages={223\ndash 252},
}

\bib{rudolph1981characterization}{incollection}{
      author={Rudolph, Daniel~J},
       title={A characterization of those processes finitarily isomorphic to a
  {B}ernoulli shift},
        date={1981},
   booktitle={Ergodic theory and dynamical systems i},
   publisher={Springer},
       pages={1\ndash 64},
}

\bib{rudolph1982mixing}{article}{
      author={Rudolph, Daniel~J},
       title={A mixing {M}arkov chain with exponentially decaying return times
  is finitarily {B}ernoulli},
        date={1982},
     journal={Ergodic Theory and Dynamical Systems},
      volume={2},
      number={1},
       pages={85\ndash 97},
}

\bib{runnels1974phase}{article}{
      author={Runnels, LK},
      author={Lebowitz, JL},
       title={Phase transitions of a multicomponent {W}idom-{R}owlinson model},
        date={1974},
     journal={Journal of Mathematical Physics},
      volume={15},
      number={10},
       pages={1712\ndash 1717},
}

\bib{salas1997absence}{article}{
      author={Salas, Jes{\'u}s},
      author={Sokal, Alan~D},
       title={Absence of phase transition for antiferromagnetic {P}otts models
  via the {D}obrushin uniqueness theorem},
        date={1997},
     journal={Journal of Statistical Physics},
      volume={86},
      number={3-4},
       pages={551\ndash 579},
}

\bib{schmidt1984invariants}{article}{
      author={Schmidt, Klaus},
       title={Invariants for finitary isomorphisms with finite expected code
  lengths},
        date={1984},
     journal={Inventiones mathematicae},
      volume={76},
      number={1},
       pages={33\ndash 40},
}

\bib{serafin2006finitary}{article}{
      author={Serafin, Jacek},
       title={Finitary codes, a short survey},
        date={2006},
     journal={Lecture Notes-Monograph Series},
       pages={262\ndash 273},
}

\bib{sinclair2017spatial}{article}{
      author={Sinclair, Alistair},
      author={Srivastava, Piyush},
      author={{\v{S}}tefankovi{\v{c}}, Daniel},
      author={Yin, Yitong},
       title={Spatial mixing and the connective constant: Optimal bounds},
        date={2017},
     journal={Probability Theory and Related Fields},
      volume={168},
      number={1-2},
       pages={153\ndash 197},
}

\bib{slawny1981ergodic}{article}{
      author={Slawny, Joseph},
       title={Ergodic properties of equilibrium states},
        date={1981},
     journal={Communications in Mathematical Physics},
      volume={80},
      number={4},
       pages={477\ndash 483},
}

\bib{smorodinsky1992finitary}{article}{
      author={Smorodinsky, Meir},
       title={Finitary isomorphism of m-dependent processes},
        date={1992},
     journal={Symbolic dynamics and its applications},
       pages={373\ndash 376},
}

\bib{soo2009translation}{article}{
      author={Soo, Terry},
       title={Translation-invariant matchings of coin-flips on {${\bf Z}^d$}},
        date={2009},
     journal={arXiv preprint math},
        ISSN={0610334/},
}

\bib{spinka2018finitary}{article}{
      author={Spinka, Yinon},
       title={Finitary coding for the sub-critical {I}sing model with finite
  expected coding volume},
        date={2018},
     journal={arXiv preprint arXiv:1801.02529},
}

\bib{steif2001thet}{article}{
      author={Steif, Jeffrey~E},
       title={The {${T,T^{-1}}$}-process, finitary codings and weak
  {B}ernoulli},
        date={2001},
     journal={Israel Journal of Mathematics},
      volume={125},
      number={1},
       pages={29\ndash 43},
}

\bib{strassen1965existence}{article}{
      author={Strassen, Volker},
       title={The existence of probability measures with given marginals},
        date={1965},
     journal={The Annals of Mathematical Statistics},
       pages={423\ndash 439},
}

\bib{stroock1992logarithmic}{article}{
      author={Stroock, Daniel~W},
      author={Zegarlinski, Boguslaw},
       title={The logarithmic {S}obolev inequality for discrete spin systems on
  a lattice},
        date={1992},
     journal={Communications in Mathematical Physics},
      volume={149},
      number={1},
       pages={175\ndash 193},
}

\bib{swendsen1987nonuniversal}{article}{
      author={Swendsen, Robert~H},
      author={Wang, Jian-Sheng},
       title={Nonuniversal critical dynamics in {M}onte {C}arlo simulations},
        date={1987},
     journal={Physical review letters},
      volume={58},
      number={2},
       pages={86},
}

\bib{timar2009invariant}{article}{
      author={Timar, Adam},
       title={Invariant matchings of exponential tail on coin flips in {${\bf
  Z}^d$}},
        date={2009},
     journal={arXiv preprint arXiv:0909.1090},
}

\bib{timar2011invariant}{article}{
      author={Timar, Adam},
       title={Invariant colorings of random planar maps},
        date={2011},
     journal={Ergodic Theory and Dynamical Systems},
      volume={31},
      number={2},
       pages={549\ndash 562},
}

\bib{weitz2004mixing}{book}{
      author={Weitz, Dror},
       title={Mixing in time and space for discrete spin systems},
   publisher={University of California, Berkeley},
        date={2004},
}

\bib{weitz2005combinatorial}{article}{
      author={Weitz, Dror},
       title={Combinatorial criteria for uniqueness of {G}ibbs measures},
        date={2005},
     journal={Random Structures \& Algorithms},
      volume={27},
      number={4},
       pages={445\ndash 475},
}

\bib{weitz2006counting}{inproceedings}{
      author={Weitz, Dror},
       title={Counting independent sets up to the tree threshold},
organization={ACM},
        date={2006},
   booktitle={Proceedings of the thirty-eighth annual acm symposium on theory
  of computing},
       pages={140\ndash 149},
}

\bib{yang1952spontaneous}{article}{
      author={Yang, Chen~Ning},
       title={The spontaneous magnetization of a two-dimensional {I}sing
  model},
        date={1952},
     journal={Physical Review},
      volume={85},
      number={5},
       pages={808},
}

\bib{yin2015spatial}{article}{
      author={Yin, Yitong},
      author={Zhang, Chihao},
       title={Spatial mixing and approximate counting for {P}otts model on
  graphs with bounded average degree},
        date={2015},
     journal={arXiv preprint arXiv:1507.07225},
}

\end{biblist}
\end{bibdiv}

\end{document}